\documentclass[10pt, letterpaper]{amsart}
\usepackage[utf8]{inputenc}
\usepackage{amsmath}
\usepackage{amsfonts}
\usepackage{amssymb}
\usepackage{graphicx}
\usepackage{slashed}
\usepackage[english]{babel}
\usepackage{amsthm}
\usepackage{mathtools}
\usepackage[top=3cm,bottom=2.5cm,left=2.5cm,right=2.5cm,marginparwidth=1.75cm]{geometry}

\usepackage{color}
\usepackage{amsmath,amsthm,amssymb}
\usepackage{graphicx}
\usepackage{color}

\numberwithin{equation}{section}

\newcommand{\Lb}{\underline{L}}

\newcommand{\beq}{\begin{equation}}\newcommand{\eq}{\end{equation}}
\newcommand{\beqs}{\begin{equation*}}\newcommand{\eqs}{\end{equation*}}
\def\pa{\partial}\def\pab{\bar\pa}

\allowdisplaybreaks[1]

\def\Lie{{\mathcal L}}

\def\f12{\frac 1 2}

\def\a{\alpha}
\def\b{\beta}
\def\om{\omega}
\def\Om{\Omega}

\def\ab{\underline{\a}}
\def\Lb{{\underline{L}}}
\def\nab{\nabla}
\def\pa{\partial}
\def\div{\text{div}}

\def\c{\cdot}
\def\pab{{\slashed{\partial}}}
\def\qq{{\langle q\rangle}}
\def\q{\mathbf{q}}

\newtheorem{thm}{Theorem}

\newtheorem{prop}{Proposition}

\newtheorem{lem}{Lemma}

\newtheorem{remark}{Remark}

\begin{document}

\title{Asymptotics and Scattering for massive Maxwell-Klein-Gordon equations}
\author{Xuantao Chen}
\address[X.C.]{Johns Hopkins University, Department of Mathematics, 3400 N.\@ Charles St., Baltimore, MD 21218, USA}
\email{xchen165@jhu.edu}

\date{}

\maketitle

\vspace{-5ex}
\begin{abstract}
    We study the asymptotic behavior and the scattering from infinity problem for the massive Maxwell-Klein-Gordon system in the Lorenz gauge, which were previously only studied for the massless system. For a general class of initial data, in particular of nonzero charge, we derive the precise asymptotic behaviors of the solution, where we get a logarithmic phase correction for the complex Klein-Gordon field, a combination of interior homogeneous function, radiation fields to null infinity and an exterior charge part for the gauge potentials. Moreover, we also derive a formula for charge at infinite time, which shows that the charge is concentrated at timelike infinity, a phenomenon drastically different from the massless case. 
    After deriving the notion of the asymptotic profile, we prove the scattering from infinity by constructing backward solutions given the scattering data. We show that one can determine the correct charge contribution using the information at timelike infinity, which is important for us to obtain solutions not only for the reduced equations in the Lorenz gauge but also for the original physical system.
\end{abstract}

\section{Introduction}
In this paper, we study both the forward (asymptotic behavior) and backward (scattering from infinity) problem of the massive Maxwell-Klein-Gordon (mMKG) equations on $\mathbb{R}^{3+1}$:\footnote{Throughout this paper, we raise and lower indices using the standard Minkowski metric $m_{\mu\nu}=\mathrm{diag}\{-1,1,1,1\}$, and we use the Einstein summation convention. Also, when the repeated index is spatial, we define the expression to be the sum regardless of whether it is upper or lower, as the spatial part of the Minkowski metric is Euclidean.}
\begin{equation}
 \label{mMKG}
\begin{split}
&\pa^\nu F_{\mu\nu}=J[\phi]_\mu=\Im(\phi \cdot\overline{D_\mu\phi}),\\
&D^\mu D_\mu\phi-m^2\phi=0.
\end{split}
\end{equation}
This is a coupled system of an electromagnetic field and a complex scalar field. We study the case $m\neq 0$, and for simplicity we can normalize the mass so that $m^2=1$.
The covariant derivative is defined as
\begin{equation*}
D_\mu =\pa_\mu+iA_\mu,
\end{equation*}
where $A_\mu$ is the connection $1$-form, and the curvature $2$-form $F_{\mu\nu}$ is given by
\begin{equation*}
F_{\mu\nu}=(dA)_{\mu\nu}=\pa_\mu A_\nu-\pa_\nu A_\mu.
\end{equation*}
The commutator of the covariant derivative on the complex scalar field is
\begin{equation}
    (D_\mu D_\nu -D_\nu D_\mu)\phi=i F_{\mu\nu}\phi.
\end{equation}
The current $1$-form $J[\phi]_\mu=\Im(\phi\, \overline{D_\mu\phi})$, where $\Im$ denotes the imaginary part of complex numbers. One can easily derive the conservation law from the equations:
\begin{equation}
    \pa^\mu J_\mu=0.
\end{equation}

The system is gauge invariant, in the sense that $(A-d\psi, e^{i\psi}\phi)$
solves the same equation \eqref{mMKG} for any potential function $\psi$. Define $\lambda=\pa^\mu A_\mu$. Then the system can be rewritten as
\begin{equation}
\begin{split}
&-\Box A_\mu=\Im(\phi \cdot\overline{D_\mu\phi})-\pa_\mu \lambda,\\
&-\Box\phi+\phi=i\lambda\phi+2iA^\mu \pa_\mu\phi-A^\mu A_\mu\phi,
\end{split}
\end{equation}
where $\Box=\pa^\mu \pa_\mu=-\pa_t^2+\triangle_x$ is the D'Alembertian operator. 

The Lorenz gauge condition says that $\lambda=\pa^\mu A_\mu=0$. Under this gauge condition, the system becomes the reduced mMKG system:
\begin{equation}\label{reducedMKG}
\begin{split}
&-\Box A_\mu=\Im(\phi \cdot\overline{D_\mu\phi}),\\
&-\Box\phi+\phi=2iA^\mu \pa_\mu\phi-A^\mu A_\mu\phi.
\end{split}
\end{equation}
One can verify that for $(A_\mu,\phi)$ satisfying \eqref{reducedMKG}, the quantity $\lambda$ satisfies its own equation
\begin{equation}\label{eqoflambda}
    \Box \lambda=|\phi|^2\lambda.
\end{equation}
Therefore, if the Lorenz gauge condition holds at an initial time slice, it holds everywhere in the spacetime. 

\vspace{1.6ex}

The global wellposedness of the mMKG equations was first established by the pioneering works \cite{eardley1982global1,eardley1982global2} of Eardly-Moncrief. Later Klainerman-Machedon \cite{KlainermanMachedon} used bilinear estimates to generalize the result to data only with bounded energy in the massless case ($m=0$). Regarding long-time behaviors, Lindblad-Sterbenz \cite{LindbladSterbenz} proved the global existence of massless MKG equations with global decay estimates, for small initial data with nonzero charge, which is previously outlined by Shu \cite{shu1991asymptotic} (see also a simplified proof by Bieri-Miao-Shahshahani \cite{bieri2014asymptotic}). The massless system was also studied for large initial data \cite{yang2016decay,yang2018global,YangYu}. For the massive case which is what we discuss in this work, the global existence with decay estimates is first studied by Psarelli \cite{psarelli1999asymptotic} with restriction on compactly supported initial data. The restriction is then removed by Klainerman-Wang-Yang \cite{KlainermanWangYang}, in which they constructed global solutions and derived decay estimates of the mMKG system in the exterior of the light cone. Then the result without assuming compactly supported data is established in Fang-Wang-Yang \cite{FangWangYang}.

\vspace{1ex}

In this paper, we study the precise asymptotic behavior of solutions of the mMKG system. We are interested in two types of problems, which are closely related: 

\begin{enumerate}
    \item \textit{Asymptotics} (forward problem): We want a precise description of the asymptotic behaviors of the solution. We show that given the initial data, one can use a set of functions to describe the asymptotic profile of the solution at (causal) infinity; we may call this set of functions the ``scattering data";
    \item \textit{Scattering from infinity} (backward problem): Now given the scattering data at infinity (satisfying some compatibility conditions), we study if there always exists a global solution of the mMKG system such that its behavior at infinity is exactly the one given by this scattering data.
\end{enumerate}

\vspace{1ex}

In \cite{CKL}, Candy-Kauffman-Lindblad studied the forward problem for the massless MKG system in the Lorenz gauge. Later He \cite{LiliHe21} established a refinement of the result, and proved a scattering from infinity result for the same system. The method of constructing backward solutions originates in the work of Lindblad-Schlue \cite{LindbladSchlue1} on the scattering from infinity for wave equations modeling the Einstein equation in wave coordinates. Recently the result in \cite{LiliHe21} was extended to large scattering data in \cite{DaiMeiWeiYang} in a gauge invariant setting.

The massive Maxwell-Klein-Gordon system, on the other hand, presents significantly different behavior compared with the massless system. Due to the presence of nonzero mass, the complex scalar field presents the behavior of the Klein-Gordon fields rather than massless waves. Therefore one needs to deal with the coupling of the electromagnetic field, which behaves similarly to wave fields, and the complex scalar field which satisfies a Klein-Gordon equation. This coupling of fields with different asymptotics
results in that many techniques used in the massless case break down. 

Unlike the wave equation, the Klein-Gordon equation does not have scale invariance, which is one of the important ingredients of the vector-field method introduced by Klainerman \cite{Klainerman1985vectorfield}. Instead, Klainerman \cite{Klainerman1985KG} found a variation of the vector-field method, which only requires the use of Lorenz boost vector fields. This is called the hyperboloidal foliation method, which is later used by LeFloch-Ma \cite{LeFLochMabook} and Wang \cite{WangEKG} in the study of the Einstein-Klein-Gordon equations. While the hyperboloidal foliation is by its nature only applicable to compactly supported data, one can combine it with estimates in the exterior as in the work \cite{KlainermanWangYang} we mentioned above (see also \cite{LeflochMafull} for the Einstein-Klein-Gordon system).

Nevertheless, the hyperboloidal version of the vector-field method itself only provides energy bound and decay estimates, hence does not directly give a precise expansion of the solution, which is the problem we are interested in. In \cite{CL}, we considered this problem for the wave-Klein-Gordon model introduced by LeFloch-Ma \cite{LeFlochMamodel} and Wang \cite{Wangmodel}, and derived both the asymptotics and scattering from infinity results. In this paper, we show that we can also give concrete answers to the forward and backward problems to the massive Maxwell-Klein-Gordon equations, in particular with nonzero charges.

\subsection{Main results}
We consider the initial value problem, which can be formulated in the Lorenz gauge by the set $(a_i,\dot a_i,\phi_0,\dot\phi_0)$. Using the equation and the gauge condition, one can then determine other components, hence the value of $A_\mu$, $\pa_t A_\mu$, $\phi$, $\pa_t\phi$ at the initial slice. In this work, we consider the initial time slice to be $\{t=t_0=2\}$. The charge is then given by
\begin{equation}
    \q_0=\frac{1}{4\pi}\int_{t=t_0} \Im(\phi\, \overline{D_0\phi})\, dx.
\end{equation}
We let the initial data satisfy similar bounds as in \cite{KlainermanWangYang,FangWangYang} (where the assumptions are written in a gauge invariant way, see \eqref{ExteriorCondition}). 

To state the results, we define the hyperboloidal coordinates $\tau=\sqrt{t^2-r^2}$, where $r=|x|$, and $y=x/t$. We will use $y$ as the variable on the hyperboloids $\{t^2-r^2=\tau^2\}$. We also define $q=r-t$, and the angular variable $\om\in\mathbb{S}^2$. The non-increasing smooth cutoff function $\chi(s)$ is defined so that $\chi(s)=1$ when $s\leq 1/2$ and $\chi(s)=0$ when $s\geq 3/4$.

\begin{thm}[Asymptotics]\label{mainthm1}
    Consider the initial value problem of \eqref{mMKG} with suitable initial data (less than $\varepsilon$ in some weighted norms), which admits global solutions by the existence result \cite{KlainermanWangYang,FangWangYang}. Then, there exist functions $a_+(y)$, $a_-(y)$, $U_\mu(y)$, defined for $|y|\leq 1$, and radiation fields $F^{I}_\mu(q,\om)$, $F^R_\mu(q,\om)$, so that in the Lorenz gauge, the solution $(A_\mu,\phi)$ satisfies
    \begin{equation*}
        \begin{split}
            A_\mu&=\frac{U_\mu(y)}\tau+\chi(\frac\qq r)\frac{F^R_\mu(q,\om)}r+A_\mu^{Linear}+O(\varepsilon\langle t+r\rangle^{-2+\delta}),\quad t-r\geq 4,\\
            A_\mu&=\frac{F^{I}_\mu(q,\om)}{r}+\q_0 r^{-1} \delta_{0\mu}+A_\mu^{Linear}+O(\varepsilon\langle t+r\rangle^{-2+\delta}),\quad t-r\leq 4,\\
            \phi&=\tau^{-\frac 32}(e^{i\tau+iU(y)\ln\tau+ih}a_+(y)+e^{-i\tau-iU(y)\ln\tau-ih}a_-(y))+O(\varepsilon\langle t+r\rangle^{-\frac 52+\delta}),
        \end{split}
    \end{equation*}
    where $U_\mu(y)$ can be determined by $a_\pm(y)$, $U(y)=U^\tau(y):=-\frac{x_\mu}\tau U_\mu(y)$, $h$ is a uniformly bounded function, and $A^{Linear}_\mu$ is the solution of a linear wave equation (with the given initial data written in the Lorenz gauge). We have the properties\footnote{We use the notation $a\lesssim b$ to denote that there exists a constant $C>0$ such that $a\leq Cb$.}
    \begin{equation*}
        |(F_\mu,F_\mu^R)(q,\om)|\lesssim \varepsilon\langle q\rangle^{-1+\delta}, \quad |a_\pm(y)|\lesssim \varepsilon (1-|y|^2)^{\frac 54-\delta}, \quad (1-|y|^2)^{-\frac 12}U_\mu(y)\rightarrow \mathcal U_\mu(\om)\text{ as $|y|\rightarrow 1$},
    \end{equation*}
    where $F_\mu(\om)$ are functions of $\om$. Moreover, we have the following relation on the radiation fields
    \begin{equation}
        rA_L\rightarrow \q_0\quad \text{as $r\rightarrow\infty$ with $q=r-t$ fixed,}
    \end{equation}
    which implies $\mathcal U_L(\om)=\q_0$ as a corollary.
    Lastly, we can express the charge using only the information at timelike infinity:
    \begin{equation}\label{q0intro}
        \q_0=\frac{1}{4\pi}\int_{|y|\leq 1} |a_-(y)|^2-|a_+(y)|^2 \, dH_1
    \end{equation}
    where $dH_1$ is the standard induced measure on the unit hyperboloid.
\end{thm}
\begin{remark}
    We extract the part $A^{Linear}_\mu$ just to improve the decay of the remainder. We can also include them in the remainder by replacing the power $-2+\delta$ by $-1-\frac{\gamma_0'-1}2$ where $\gamma_0'>1$ comes from the initial data.
\end{remark}
\begin{remark}
    The result says that $A_\mu\sim U_\mu(y)/\tau$ towards timelike infinity, and $\phi$ behaves like a linear expansion with a logarithmic phase correction. Towards null infinity, all terms (apart from the remainder) in the expansion contribute to radiation fields, so neither $F^R_\mu$ nor $F^I_\mu$ is the total radiation field of $A_\mu$. In fact, $F_\mu^R$ is the radiation field of lower order terms in the interior, and $F^I_\mu$ is $F_\mu^R$ combined with the radiation fields that originate from the interior leading behavior $U_\mu(y)/\tau$.
\end{remark}

\vspace{1ex}
We now state our next result, in which we study the backward problem. We define the scattering data by the set $(a_+(y),a_-(y),F_\mu(q,\om))$. We impose the decay conditions
\begin{equation}
    |(1-|y|^2)^{|I|}\nab_y^I a_\pm(y)|\lesssim \varepsilon(1-|y|^2)^\a,\quad |(\langle q\rangle \pa_q)^k \pa_\om^\b F_\mu(q,\om)|\lesssim \varepsilon\langle q\rangle^{-1+\gamma},\quad |I|\leq N,\, k+|\b|\leq N
\end{equation}
for some $N\geq 5$, $\a\geq \frac 74$, $\gamma<1/2$. Motivated by \eqref{q0intro}, we define the charge at infinity $\q_\infty$ by
\begin{equation}\label{defqinftyintro}
        \q_\infty :=\frac{1}{4\pi}\int_{|y|\leq 1} |a_-(y)|^2-|a_+(y)|^2 \, dH_1
    \end{equation}
In the forward problem, we have that towards null infinity $rA_L\rightarrow \q_0$. This is in fact from the Lorenz gauge condition, so here we need to have a similar property without assuming the condition. We will prove the following important observation:
\begin{prop}\label{propintro}
    Given $a_\pm(y)$, we determine functions $U_\mu(y)$ by \footnote{Of course, there is a smoothness problem for the solution when $t=r$, but it does not matter in view of the strong Huygens' principle.}
    \begin{equation}
        -\Box(\frac{U_\mu(y)}\tau)=\frac{x_\mu}\tau \tau^{-3}(|a_+(y)|^2-|a_-(y)|^2),
    \end{equation}
    which is equivalent to an elliptic equation on the unit hyperboloid. Then we have
    \begin{equation}
        r\frac{U_\mu(y)}\tau \rightarrow \q_\infty,\quad\text{as $r\rightarrow\infty$, with $q=r-t<0$ fixed,}
    \end{equation}
    where $\q_\infty$ is defined by $a_\pm(y)$ by \eqref{defqinftyintro}.
\end{prop}
We note that while in the forward problem, the analogous relation is just a corollary of the Lorenz gauge condition, in the backward problem one needs to prove it for arbitrary $a_\pm(y)$ without assuming the gauge condition. This is a crucial step in our proof.


We now define the approximate solutions:
\def\az{A^{(0)}}
\def\phiz{{\phi^{(0)}}}
\begin{equation}
\begin{split}
    \az_\mu&=A^M_\mu+\chi(\frac\qq r) \left(\frac{F_\mu(q,\om)}r+\frac{F^{(1)}_\mu(q,\om)}{r^2}\right)+\q_\infty \chi_{ex}(q) \delta_{0\mu} r^{-1},\\
    \phiz&=\tau^{-\frac 32}\left(e^{i\tau+i\int (\az)^\tau d\tau} a_+(y)+e^{-i\tau-i\int (\az)^\tau d\tau} a_-(y)\right),
    \end{split}
\end{equation}
where $A^M_\mu$ is exactly $U_\mu(y)/\tau$ when $t-r\geq 4$ and suitably modified near the light cone, and $F_\mu^{(1)}$ is a second-order approximation introduced for technical reasons. We then impose the condition on $F_\mu$, which can be viewed as the Lorenz gauge condition at null infinity:
\begin{equation}\label{admissiblecondintro}
    F_L+F^M_L+\q_\infty \chi_{ex}(q)=\q_\infty
\end{equation}
where $F^M_\mu$ is the radiation field of $A^M_\mu$. Note that it would not be possible to impose this without Proposition \ref{propintro}, because $F_\mu$ decay to zero as $q\rightarrow -\infty$. In fact, by Proposition \ref{propintro} we know that $F^M_L=\q_\infty$ when $q=r-t\leq -4$. Therefore $F_L$ is actually compactly supported.

We are now ready to state our next result.
\begin{thm}[Scattering from infinity]\label{mainthm2}
    Consider a set of scattering data $(a_+(y),a_-(y),F_\mu(q,\om))$ which satisfies the condition \eqref{admissiblecondintro}. Then, there exists a global solution to the reduced mMKG system with the exact same asymptotic behavior described by the scattering data, which satisfies the Lorenz gauge condition, hence also a solution to the original mMKG system. More precisely, for the approximate solution $(\az_\mu,\phiz)$ defined above, there exists $(v_\mu,w)$ such that $(\az_\mu+v_\mu,\phiz+w)$ solves the mMKG system, with $v_\mu$ and $w$ being lower order terms.
\end{thm}
We note that while the forward result gives the asymptotic behavior of the solution using $a_\pm(y)$ and radiation fields, it does not answer the question of whether arbitrarily given scattering data admits a global solution, and it is the backward result that gives the affirmative answer.

\subsection{Related works}
We also mention other related results on equations apart from the massive or massless Maxwell-Klein-Gordon equations.

In the forward direction, the radiation field of a wave equation is first studied by Friedlander \cite{F62}. The asymptotics of nonlinear wave equations were studied in \cite{BB15,BVW15}, and the Einstein equation in higher dimensions in \cite{W14,W13}. In \cite{L17}, Lindblad studied the asymptotics of the $1+3$ dimensional Einstein equation in wave coordinates with a logarithmically modified behavior at null infinity and also an interior asymptotics. In recent years, the asymptotics of the quasilinear wave equation were also studied in \cite{DengPusateri,Yuforward}. Polyhomogeneous expansions were derived by Hintz and Vasy \cite{HVMin}, for the Einstein equation in generalized wave coordinates. 
For the linear Klein-Gordon equation, the asymptotics can be found in \cite{H97}. For nonlinear Klein-Gordon equations in one spatial dimension, the asymptotic behavior is derived by Delort \cite{D1,D2} (see also \cite{LindbladSoffer05forward}). With variable coefficients on nonlinearities, the problem was studied by Lindblad-Soffer \cite{LS15} and
Sterbenz \cite{sterbenz2016dispersive}, with further results in \cite{LLS20,LLS21,LLSS22,LuhrmannSchlag21}.

For the backward problem, i.e. scattering from infinity, Flato-Simon-Taflin \cite{flato1987global} proved the existence of the modified wave operators for the Maxwell-Dirac equations. Lindblad-Soffer \cite{LindbladSoffer05backward} studied the scattering from infinity for 1D nonlinear Klein-Gordon equations. The scattering problem for blackhole spacetimes was studied in \cite{Dafermos2013scattering,Dafermos2018scattering,angelopoulos2020non}. Yu \cite{Yubackward} established the existence of modified wave operators for a scalar field equation satisfying the weak null condition. Apart from \cite{LindbladSchlue1}, Lindblad-Schlue recently studied scattering from infinity for wave equations with sources modeling slow decay along the light cone \cite{LindbladSchlue2}.

We also note that one can study similar problems on the Fourier side. For the related wave-Klein-Gordon models, Ionescu-Pausader \cite{IPmodel} showed the global existence using the spacetime resonance method and also derived the modified scattering (in the forward direction; then they also did it for the Einstein-Klein-Gordon equations \cite{IPEKG}). Later Ouyang \cite{Ouyang} studied the corresponding backward problem.

\subsection{Idea of proof}
\subsubsection{Global existence}
We note two important results on the global existence and decay estimates of solutions of the mMKG equations \cite{KlainermanWangYang,FangWangYang}. The work \cite{KlainermanWangYang} constructs solutions in the exterior of a light cone. This resolves the problem that previously, the hyperboloidal foliation method can only deal with compactly supported initial data. With the existence and estimates of the exterior solution, one gets control of the energy on the light cone which makes it possible to apply the hyperboloidal foliation in the interior. The interior part is done in \cite{FangWangYang} (see also an earlier work on the interior assuming compact initial data \cite{psarelli1999asymptotic}).

\vspace{0.5ex}
In this paper, we shall make use of the results to get preliminary decay estimates which will be used when deriving asymptotic behaviors. Since the Klein-Gordon field concentrates more in the interior, the corresponding work \cite{FangWangYang} is more relevant to our work. (Nevertheless, we note that the result in \cite{KlainermanWangYang} also plays a fundamental role in making it possible for us to study this much more general, non-compactly supported initial data.) We will give a brief review of the estimates in the interior. The only adaptation we make here is that we do the calculation of higher order commutators and higher order estimates of the flux on the null cone so that we can apply them to our case, where we need more than two orders of vector fields.

\subsubsection{Asymptotics}
In the Lorenz gauge $\pa^\mu A_\mu=0$, the mMKG system becomes
\begin{equation*}
\begin{split}
-\Box A_\mu&=\Im(\phi \cdot\overline{D_\mu\phi}),\\
-\Box\phi+\phi&=2iA^\mu \pa_\mu\phi-A^\mu A_\mu\phi.
\end{split}
\end{equation*}
This is a wave-Klein-Gordon type coupled system. With the assumption of compactly supported initial data, we studied a similar model in \cite{CL}.
Here due to the effect of charge, generally one cannot expect compactly supported initial data. However, we will see that one still has a relatively good decay estimate of the complex scalar field $\phi$ near the light cone, which serves, together with the initial data in the interior, as an ``initial condition" for us to derive the asymptotics of $\phi$ in the interior. 

To derive the asymptotics, we first decompose the Klein-Gordon equation. 
The main term on the right hand side is $\sim 2i A^\tau \pa_\tau\phi$, where $A^\tau=-\frac{x_\mu}\tau A_\mu$ and $\pa_\tau=\frac{x^\mu}\tau \pa_\mu$, since derivatives tangent to the hyperboloids decay better, and cubic terms are ignorable. So approximately, using $-\Box=\pa_\tau^2+3\tau^{-1}\pa_\tau-\tau^{-2}\triangle_y$, we have for $\Phi=\tau^\frac 32 \phi$ that
\begin{equation*}
    \pa_\tau^2 \Phi+\Phi=2i A^\tau \pa_\tau \Phi+\cdots
\end{equation*}
where the remainder decays better in $\tau$.
Now define $\Phi_\pm=e^{\mp i\tau} (\pa_\tau\Phi\pm i\Phi)$, so $\pa_\tau \Phi=\frac 12(e^{i\tau}\Phi_++e^{-i\tau}\Phi_-)$. Therefore
\begin{equation*}
    \pa_\tau \Phi_\pm=e^{\mp i\tau} i A^\tau (e^{i\tau}\Phi_++e^{-i\tau}\Phi_-),
\end{equation*}
so for example for $\Phi_+$ we get $\pa_\tau \Phi_+=iA^\tau \Phi_+$ plus oscillating part. The oscillating part has extra cancellation after integration, so it also affects little. Therefore we expect $\Phi_+$ to be $e^{i\int A^\tau d\tau}$ multiplied by a function of $y$ (and similarly for $\Phi_-$), and hence obtain
\begin{equation*}
    \begin{split}
        \phi&\sim \tau^{-\frac 32}(e^{i\tau+i\int A^\tau d\tau} a_+(y)+e^{-i\tau-i\int A^\tau d\tau} a_-(y)),\\
        \pa_\tau \phi &\sim i\tau^{-\frac 32}(e^{i\tau+i\int A^\tau d\tau} a_+(y)-e^{-i\tau-i\int A^\tau d\tau} a_-(y)).
    \end{split}
\end{equation*}

One can plug this expansion in the right hand side of the wave equations. In \cite{CL} we studied nonlinearities like $\phi^2$ and $(\pa_t\phi)^2$, and the expansion gives
\begin{equation*}
    2\tau^{-3} a_+(y)a_+(y)+\tau^{-3}e^{2i\tau+2i\int A^\tau d\tau}(a_+(y))^2+\tau^{-3}e^{-2i\tau-2i\int A^\tau d\tau}(a_-(y))^2,
\end{equation*}
where we see that the last two terms are oscillating. Then we performed an integration by part argument using the representation formula, which requires a tedious calculation. Remarkably, the special structure of the nonlinearity $\Im(\phi\, \overline{D_\mu\phi})$ here ensures that there will be no such oscillating terms: Modulo lower order terms, the nonlinearity $\sim -\frac{x_\mu}\tau\Im(\phi\, \overline{\pa_\tau\phi})$, and plugging in the expansions we get
\begin{equation*}
    \Im(\phi\, \overline{\pa_\tau\phi})\sim \tau^{-3}(|a_-(y)|^2-|a_+(y)|^2).
\end{equation*}
Therefore, the main terms on the right hand side of the wave equations of $A_\mu$ are $-\frac{x_\mu}\tau \tau^{-3} (|a_-(y)|^2-|a_+(y)|^2)$.\footnote{We remark that for the wave-Klein-Gordon model, the pair $a_\pm(y)$ are in fact conjugate to each other. Here this is nonzero because $\phi$ is complex.} These are still of the form $-\Box \psi=\tau^{-3} Q(y)$, which is studied in \cite{CL}: One can find a function $\Psi(y)$ which satisfies the equation $$-\Box(\frac{\Psi(y)}\tau)=\tau^{-3} Q(y),$$
or equivalently, an elliptic equation on the unit hyperboloid $$\triangle_y\Psi+\Psi=Q(y).$$
In practice, to avoid the jump along the light cone, we proceed in this way: we start solving the wave equation at a positive time, say $t=2$, with the source $\tau^{-3}Q(y)$. One can then show that when $t-r\geq 4$, the solution exactly equals $\Psi(y)/\tau$. This in fact comes from the strong Huygens' principle.  
Applying this to the equations of $A_\mu$, we see that
\begin{equation*}
    A_\mu\sim U_\mu(y)/\tau
\end{equation*}
in the interior. This also gives the main part of the phase correction on the expansion of the scalar field: $\int A^\tau d\tau\sim U(y)\ln\tau$, where $U(y)=U^\tau(y)$. 

For the wave fields, we still need to describe the behaviors at null infinity. It is not hard to show the existence of radiation fields, defined as the limit of $rA_\mu$ as $r\rightarrow\infty$ with $q=r-t$ and the angular variable $\om$ fixed. However, there are additional structures. If we decompose the Lorenz gauge condition $\pa^\mu A_\mu=0$ in the null frame, we see that $\pa_q A_L$ behaves like derivatives tangent to the outgoing light cone, which are expected to decay better towards null infinity. This means that near null infinity we have
\begin{equation*}
    \pa_q A_L\sim 0.
\end{equation*}
Note that due to the charge effect, there is a $\q_0 r^{-1}$ tail of $A_\mu$ (we can make it only on $A_0$) at spacelike infinity. The above relation implies that, this effect on the component $A_L$ will be preserved along the null infinity, all the way to the
timelike infinity. One can show that
\begin{equation*}
    rA_L\rightarrow \q_0\quad \text{as $r\rightarrow\infty$, for every $q=r-t$ and $\om$ fixed.}
\end{equation*}
We see that the $L$ component of the radiation fields must be $\q_0$, which is the charge of the system given by the initial data. This structure comes from the gauge condition itself, and was first observed in \cite{CKL} for the massless MKG system. 

However, the charge itself behaves rather differently for the massive MKG system. In the massless system, the scalar field satisfies a wave equation, and then mainly propagates along the null directions. Using the conservation law $\pa^\mu J_\mu$, one can express the charge using the asymptotics of the scalar field at null infinity. For the massive field we study here, it decays much more towards null infinity, and one can verify that the flux is zero at null infinity when we integrate the divergence identity. This is not a contradiction to the conservation law, but the reflection of the fact that the charge is concentrated at timelike infinity. We can show that
\begin{equation*}
    \q_0=\int_{|y|\leq 1} |a_-(y)|^2-|a_+(y)|^2\, dH_1.
\end{equation*}
This will also be extremely important for the backward problem.

\subsubsection{Scattering from infinity}
We now discuss the backward problem. We shall study the problem for the reduced mMKG system first, and then show that the constraint we impose on the scattering data ensures that the solution we construct also solves the original mMKG system.

First, for scattering data, we should give the pair $a_\pm(y)$, which determines the size of the Klein-Gordon (scalar) field. It should also give the leading behavior of the nonlinearities that appeared in the wave equation of $A_\mu$. In fact, as in the forward problem, we can consider a part $A_\mu^M$ satisfying the equation
\begin{equation*}
    -\Box A_\mu^M=\frac{x_\mu}\tau \tau^{-3}(|a_+(y)|^2-|a_-(y)|^2)
\end{equation*}
with vanishing data at $\{t=2\}$. Then we have $A^M_\mu=U_\mu(y)/\tau$ for some $U_\mu(y)$ which should be the leading behavior of $A_\mu$ towards timelike infinity. The other piece of scattering data is the radiation field $F_\mu(q,\om)$. We know that $A^M_\mu$ contributes to the radiation field, so what we impose should be from the remaining part, and we expect (and in fact, also need) them to decay in $\langle q\rangle$. Then we might be tempted to define the approximate solution as 
\begin{equation*}
    \az_\mu=A^M_\mu+\chi(\frac \qq r)\frac{F_\mu(q,\om)}r,\quad \phi=\tau^{-\frac 32}(e^{i\tau+i\int (\az)^\tau d\tau} a_+(y)+e^{-i\tau-i\int (\az)^\tau d\tau} a_-(y)),
\end{equation*}
which is similar to \cite{CL}, so one can apply the method used in that work to get a global solution that scatters to the asymptotic behavior given by $(a_+(y),a_-(y),F_\mu(q,\om))$. However, the solution we get in this way will not have the $r^{-1}$ tail at spacelike infinity, hence not exactly what we want. 

Now since $\chi_{ex}(q)r^{-1}$ is an exact solution of the linear wave equation, we can then consider adding a part $C_{ex}\chi_{ex}(q)r^{-1}$, where $C_{ex}$ is a constant, $\chi_{ex}(q)$ a cutoff equal to $1$ when $q\geq -1/2$ and $0$ when $q\leq -1$, to the approximate solution, say the $\az_0$ component, so that there is a part that represents the effect of charge. We see that one can find backward solutions of the reduced mMKG system for any $C_{ex}$. However, we still need the Lorenz gauge condition to hold, which requires the approximate solutions to satisfy this condition approximately. Recall that the Lorenz gauge condition written in the null frame reads $\pa_q A_L\sim 0$, so similar to the forward problem, we expect the $L$ component of the (total) radiation fields to be conserved along $\pa_q$ at null infinity, this time from the timelike infinity to the spacelike infinity. Therefore, $C_{ex}$ should be determined from the information at timelike infinity.

This raises one of the key questions in this work: What is the $L$ component of the radiation fields of $A^M_\mu$, or equivalently $U_\mu(y)/\tau$, as $q\rightarrow -\infty$? Recall that $U_\mu(y)$ is determined by $a_\pm(y)$ by a wave equation or equivalently an elliptic equation on the unit hyperboloid. It is not hard to see that the radiation field of $U_\mu(y)/\tau$ is independent of $q$, and equals
$$\lim_{|y|\rightarrow 1} (1-|y|^2)^{-\frac 12} U_\mu(y),$$
so what we are asking is
$$\lim_{|y|\rightarrow 1} (1-|y|^2)^{-\frac 12}U_L(y)=\, ?$$

Up to now, it is not clear at all whether this limit as $q=r-t\rightarrow -\infty$ will be independent of the angular variable $\om$, but we need it to be because otherwise with the exterior part being $C_{ex}\chi_{ex}(q)r^{-1}$, there is no hope that the Lorenz gauge condition will be true. We however show in this paper, through a careful calculation of the limiting measure, that the limit, is indeed a constant independent of $\om$, and is nothing but the integral
\begin{equation*}
    \frac{1}{4\pi}\int_{|y|\leq 1} |a_-(y)|^2-|a_+(y)|^2 \, dH_1
\end{equation*}
which is the same expression as in the relation derived in the forward problem (of course, the $a_\pm(y)$ there is determined from initial data, and here $a_\pm(y)$ is where we start). This is the result of Proposition 1.

We denote the value of this integral as $\q_\infty$ to represent the charge computed at infinity. Then we are ready to write down the corrected expression of $\az$:
\begin{equation}
    \az_\mu=A^M_\mu+\chi(\frac\qq r) \left(\frac{F_\mu(q,\om)}r+\frac{F^{(1)}_\mu(q,\om)}{r^2}\right)+\q_\infty \chi_{ex}(q) \delta_{0\mu} r^{-1},
\end{equation}
where the $F^{(1)}$ is a second order approximation for technical reasons. Now, to make the approximate Lorenz gauge condition at null infinity hold, it suffices to impose the constraint that the $L$ component of the radiation fields of $\az_\mu$ is always equal to $\q_\infty$. Since when $q\leq -4$ and $q\geq 2$ it is already $\q_\infty$ without $F_L$, this constraints implies that $F_L$ is compactly supported in $-4\leq q\leq 2$.

While we have discussed the Lorenz gauge condition towards null infinity in detail, and argued that $\az_\mu$ along with $\phiz$ are now a good approximate solution for both the reduced and original system, we still have to ensure that the condition is also approximately satisfied in the interior. Since the $A^M_\mu$ (or $U_\tau(y)/\tau$) are the leading terms in the interior, we should look at the behavior $\pa^\mu A^M_\mu$ towards timelike infinity. To do this we can commute the wave equations of $A^M_\mu$ with $\pa^\mu$ (i.e. $-\pa_t$, $\pa_i$). Since $\pa_\mu x^\mu=4$, $x_\mu \pa^\mu=\tau \pa_\tau$, we have
\begin{equation*}
\begin{split}
    \Box \pa^\mu A^M_\mu&=(\pa^\mu x_\mu)\tau^{-4} (|a_-(y)|^2-|a_+(y)|^2)+x_\mu \pa^\mu (\tau^{-4} (|a_-(y)|^2-|a_+(y)|^2))\\
    &=4\tau^{-4} (|a_-(y)|^2-|a_+(y)|^2)-\tau\c 4\tau^{-5} (|a_-(y)|^2-|a_+(y)|^2)=0,
\end{split}
\end{equation*}
so $\pa^\mu A^M_\mu$ satisfies a linear homogeneous wave equation. Now since at $t=2$ we have $\pa^\mu A^M_\mu$ compactly supported in $\{|x|\leq 2\}$, by strong Huygens' principle we see that $\pa^\mu A^M_\mu$ is supported in $\{0\leq t-r\leq 4\}$, a region that should be categorized as near light cone rather than interior. Therefore the Lorenz gauge condition is also reasonably satisfied in the interior. Then we can use similar backward energy estimates introduced in \cite{CL} to obtain a backward solution, which can be shown to also satisfy the Lorenz gauge condition.

\subsection*{Acknowledgements}The author would like to thank his advisor, Hans Lindblad, for many helpful discussions and his encouragement on studying this problem.

\section{Notations, Preliminaries}
\subsection{Hyperboloidal coordinates}
We introduce the hyperboloidal coordinates $\tau=\sqrt{t^2-r^2}$ where $r=|x|$, and $y=x/t$, so $|y|=|x|/t$. We also define the angular variables $\om\in\mathbb{S}^2$ and $\om_i=x_i/|x|$. The $\tau$-hyperboloid is defined by $H_\tau=\{(t,x)\colon t^2-|x|^2=\tau^2\}$. 

In $(\tau,y_i)$ coordinates, we have\footnote{We use the Einstein summation convention. Also, when the repeated index is spatial, we define the expression to be the sum regardless of whether it is upper or lower, as the spatial part of the Minkowski metric is Euclidean.} \[\pa_\tau=\frac{t}{\tau}\pa_t+\frac{x^i}\tau \pa_i,\quad \pa_{y_i}=\frac{\tau}{(1-|y|^2)^{\frac 32}}\left((y_i y_j+\delta_{ij})\pa_j+y_i \pa_t\right).\]
We also have the dual frame 
\[d\tau=\frac t\tau dt-\frac{x^i}\tau dx^i,\quad dy^i=-\frac{x_i}{t^2} dt+\frac 1t dx^i.\]
Then, for a vector field $V^\mu$, we can define the quantities $V^\tau$ and $V^{y_i}$. For example, in this paper we frequently use $A^\tau:=-\frac{x_\mu}\tau A_\mu$.

We define the hyperboloidal orthonormal frame $\{\pa_\tau,\tau^{-1}\Om_{0r},e_1,e_2\}$ where $\Om_{0r}=\om^i \Om_{0i}$, and $\{e_1,e_2\}$ is a set of orthonormal frames on the $2$-sphere.

We will also use truncated versions of hyperboloids in this paper: $\hat H_\tau=H_\tau\cap \{t-r\geq 1\}$, and $\widetilde H_\tau=H_\tau\cap \{t-r\geq r^\frac 12\}$.

\subsection{Vector fields}We use $Z$ to denote the vector fields in the set \[\mathcal{Z}=\{\pa_\mu,\Omega_{0i}=t\pa_i+x_i\pa_t,\Omega_{ij}=x_i\pa_j-x_j\pa_i\}.\]
It is straightforward to verify the schematic commutation relation $[Z,Z]=\pm Z$. We also use $\Om$ to denote boost and rotation vector fields.
Sometimes we also use the scaling vector fields $S=t\pa_t+x^i\pa_i$. We use the notation $(Z,S)$ to denote the set above added with $S$.


We use the multi-index notation: for $I=(\b_1,\b_2,\cdots,\b_n)$, we define $$Z^I=Z_1^{\b_1} Z_2^{\b_2} \cdots Z_n^{\b_n}$$ where $Z_1,\cdots Z_n \in \mathcal{Z}$, The size of a multi-index $I=(\beta_1,\beta_2,\cdots,\beta_n)$ is defined as $|I|:=\sum_{i=1}^n \beta_i$.

Note that the rotation and boost vector fields are tangent to the hyperboloids $H_\tau$, and we have the following relation between $\pa_{y_i}$ and vector fields:
\begin{equation}\label{dy}
    \pa_{y_i}=\frac 1{1-|y|^2} \left(\Omega_{0i}+\frac{x_j}{t} \Omega_{ij}\right).
\end{equation}
Moreover, the Laplacian-Beltrami operator on the unit hyperboloid $\triangle_y$ can be expressed by vector fields:
\begin{equation}\label{laplaciany}
    \triangle_y=\sum_{i=1}^3 \Omega_{0i}^2+\sum_{j=1}^3 \sum_{i=1}^j \Omega_{ij}^2.
\end{equation}
We also use $\triangle_x$ and $\triangle_\om$ to denote the Laplacian on $\mathbb{R}^3$ and $\mathbb{S}^2$ respectively.


\subsection{Integration}
Apart from the standard integration with respect to induced volume elements, it is convenient to consider the integral of a function $u(t,x)$ on hyperboloids with respect to the measure $dx$.
\[\int_{H_\tau} u \, dx:=\int_{\mathbb{R}^3}u(\sqrt{\tau^2+|x|^2},x)\, dx,\]
In this paper, we define the $L^2$ norm on the hyperboloid using the measure $dx$ instead:
\[||u||_{L^2(H_\tau)}=\left(\int_{H_\tau} |u|^2\, dx\right)^\frac 12.\]
We also use similar definitions for the $L^2$ norm on truncated hyperboloids $\hat H_\tau$, $\widetilde H_\tau$.

\subsection{Null frame}
Besides the hyperboloidal frame, we also need the null frame, which is fundamental in the analysis of wave equations. It is defined by the set $\{L,\Lb,e_1,e_2\}$, where $L=\pa_t+\pa_r$, $\Lb=\pa_t-\pa_r$, with $\pa_r=\om^i\pa_i$, and $\{e_1,e_2\}$ is an orthonormal frame on the sphere. We also extensively use the function $q=r-t$, and the derivative $\pa_q=-\frac 12(\pa_r-\pa_t)$. We also use the tangential derivatives $\pab_i:=\pa_i-\om_i\pa_r$.

\subsection{Lie derivatives}
We define the Lie derivative of tensor fields with respect to a vector field by
\begin{multline}
\Lie_Z K^{\a_1\a_2\cdots \a_s}_{\b_1\b_2\cdots \b_r}=Z\, K^{\a_1\a_2\cdots \a_s}_{\b_1\b_2\cdots \b_r}-(\pa_\nu Z^{\a_1})\, K^{\nu\a_2\cdots \a_s}_{\b_1\b_2\cdots \b_r}-\cdots-(\pa_{\nu} Z^{\a_s})\, K^{\a_1\a_2\cdots \nu}_{\b_1\b_2\cdots \b_r}+(\pa_{\b_1}Z^\nu)\,  K^{\a_1\a_2\cdots \a_s}_{\nu\b_2\cdots \b_r}\\
    +\cdots+(\pa_{\b_r} Z^\nu)\, K^{\a_1\a_2\cdots \a_s}_{\b_1\b_2\cdots \nu}.
\end{multline}
In particular, if $X$, $Y$ are vector fields, we have $\Lie_X Y=[X,Y]$. In this paper, since the background metric is Minkowski, and the vector fields we use will be Killing, the Lie derivatives commute with contractions (we will see later that they also commute with partial derivatives). The factors of the form $\pa Z$ in the definition above will also be constants.

\subsection{Curvature components}
We can write components of the curvature $2$-form $G$ in the null frame, denoted as
\begin{equation}
    \a_B[G]:=G_{Le_B},\quad \ab_B[G]:=G_{\Lb e_B},\quad \rho[G]:=\frac 12 G_{L\Lb},\quad \sigma[G]:=G_{e_1 e_2},\quad B=1,2.
\end{equation}
We can also consider its components in the hyperboloidal frame. We have
\begin{equation}
    G(\pa_\tau, \tau^{-1}\Om_{0r})=\rho[G],\quad G(\pa_\tau,e_B)=\frac{t+r}\tau \a[G]+\frac{t-r}\tau \ab[G],\quad G(\tau^{-1}\Om_{0r},e_B)=\frac{t+r}\tau \a[G]-\frac{t-r}\tau \ab[G].
\end{equation}
Later, as in \cite{FangWangYang}, we will see that the components in the hyperboloidal frame have similar bounds, so we denote all components in this frame by $G^H$.

\section{The existence results}
In this section, we brief review the results in \cite{KlainermanWangYang,FangWangYang} with slight adaptations to our case. This part is done in a gauge independent way.
\subsection{Higher order commutators}
Here we view $D_X\phi=X^\mu D_\mu\phi$ as a scalar field.\footnote{In some literatures, the notation $D_X D_Y \phi$ could mean the $X,Y$-component of the tensor $D_\mu D_\nu\phi$, but here it represents $D_X(Y^\mu D_\mu \phi)$.} Define the trilinear form
\begin{equation*}
    Q(F,\phi,Z):=2i Z^\nu F_{\mu\nu} D^\mu \phi+i\pa^\mu(Z^\nu F_{\mu\nu}) \phi.
\end{equation*}
We have 
\begin{equation*}
    [\Box_A,D_Z]\phi=Q(F,\phi,Z),
\end{equation*}
\begin{equation*}
    [\Box_A,D_X D_Y]\phi=Q(F,D_X \phi,Y)+Q(F,D_Y \phi,X)+Q(\Lie_X F,\phi,Y)+Q(F,\phi,[X,Y])-2F_{X\mu} F^{\mu}_{\, \, Y}\phi.
\end{equation*}
The key of the proof is the following relation
\begin{equation*}
    \begin{split}
        D_X Q(F,\phi,Y)&= D_X (2i Y^\nu F_{\mu\nu} D^\mu \phi+i\pa^\mu(Y^\nu F_{\mu\nu}) \phi)\\
        &=i X^\lambda (\pa_\lambda+iA_\lambda)(2Y^\nu F_{\mu\nu} D^\mu\phi+\pa^\mu(Y^\nu F_{\mu\nu})\phi)\\
        &=2i(\Lie_X Y)^\nu F_{\mu\nu}D^\mu\phi+2i Y^\nu \Lie_X F_{\mu\nu} D^\mu\phi+2i Y^\nu F_{\mu\nu} (\Lie_X+iA_X)D^\mu\phi\\
        &\quad +i\Lie_X (\pa^\mu (Y^\nu F_{\mu\nu}))\phi+i\pa^\mu(Y^\nu F_{\mu\nu})D_X \phi\\
        &=2i(\Lie_X Y)^\nu F_{\mu\nu}D^\mu\phi+2i Y^\nu \Lie_X F_{\mu\nu} D^\mu\phi+2iY^\nu F_{\mu\nu} D^\mu D_X\phi-2F_{X\mu}F^{\mu}_{\,\, Y}\phi\\
        &\quad +i\pa^\mu ((\Lie_X Y^\nu) F_{\mu\nu})\phi+i\pa^\mu(Y^\nu\Lie_X F_{\mu\nu})\phi+i\pa^\mu(Y^\nu F_{\mu\nu})D_X\phi\\
        &=Q(F,\phi,[X,Y])+Q(\Lie_X F,\phi,Y)+Q(F,D_X\phi,Y)-2F_{X\mu}F^{\mu}_{\,\, Y}\phi
    \end{split}
\end{equation*}
where we used
\begin{equation}\label{lieDcommutator}
    \begin{split}
        (\Lie_X+iA_X) D^\mu\phi&=D_X D^\mu\phi-D^\lambda\phi \pa_\lambda X^\mu=D^\mu (D_X\phi)+iF_{X}^{\, \, \mu}\phi-(\pa^\mu X^\lambda)D_\lambda \phi-D^\lambda\phi \pa_\lambda X^\mu\\
        &=D^\mu D_X\phi+iF_{X}^{\, \, \mu}\phi
    \end{split}
\end{equation}
since $X$ is Killing.

We can further calculate the third-order commutator
\begin{equation*}
    \begin{split}
        [\Box_A, D_X D_Y D_Z]\phi&=\Box_A D_X D_Y D_Z \phi-D_X \Box_A D_Y  D_Z \phi+D_X \Box_A D_Y D_Z\phi -D_X D_Y D_Z \Box\phi\\
        &=[\Box_A, D_X]D_Y D_Z\phi+D_X [\Box_A,D_Y D_Z]\phi\\
        &=Q(F,D_Y D_Z\phi,X)+D_X ([\Box_A,D_Y D_Z]\phi)
    \end{split}
\end{equation*}
Using the formula for the second-order commutator and the relation above, we see that the third-order commutator is of a similar form. Inductively, using $\Lie_Z Z=\pm Z$, one can show the schematic relation (where we omit all constant factors)
\begin{equation*}
        [\Box_A, D_Z^I]\phi= \sum_{|I_1|+|I_2|\leq |I|} Q(\Lie_Z^{I_1} F,D_Z^{I_2}\phi,Z)
        +\sum_{|I_3|+|I_4|+|I_5|<|I|} Z^\lambda \Lie_Z^{I_3} F_{\lambda\mu} \Lie_Z^{I_4} F^{\mu}_{\, \, \nu} Z^\nu D_Z^{I_5}\phi.
\end{equation*}
Note that since $Z$ are Killing, $\Lie_Z m=0$ and $\Lie_Z$ commutes with raising and lowering indices.


We also need to deal with the Lie derivatives of the current $1$-form. Using \eqref{lieDcommutator} we have
\begin{equation*}
    \Lie_Z J_\mu[\phi]=\Im (D_Z\phi\c\overline {D_\mu\phi})+\Im(\phi\c\overline{D_\mu D_Z\phi+i F_{Z\mu}\phi}).
\end{equation*}
We can use this inductively and show that schematically,
\begin{equation*}
\begin{split}
    \Lie_Z^I J_{\mu}[\phi]&=\sum_{|I_1|+|I_2|\leq |I|}\Im(D_Z^{I_1} \phi\c\overline{D_\mu D_Z^{I_2} \phi})+\sum_{|J|<|I|}\Lie_Z^{J} \Im(\phi\cdot\overline{iF_{Z\mu}\phi})\\
    &=\sum_{|I_1|+|I_2|\leq |I|}\Im(D_Z^{I_1} \phi\cdot\overline{D_\mu D_Z^{I_2} \phi})+\sum_{|I_3|+|I_4|+|I_5|<|I|}\Im(D_Z^{I_3}\phi\cdot\overline{i\Lie_Z^{I_4} F_{Z\mu} D_Z^{I_5}\phi})
    \end{split}
\end{equation*}
using $[Z,Z]=\pm Z$.

\subsection{Null cone flux}
We expect a part of $F$ governed by the charge. The chargeless part of $F$ is defined by
\begin{equation*}
    \tilde F=F-\q_0 r^{-2} \chi_{\{t-r>2\}} dt\wedge dr.
\end{equation*}
This decomposition is crucial in the exterior, but less important in the interior, where one may just do the estimate for the original Maxwell field $F$. However, the energy bounds from the exterior are all on the chargeless part $\tilde F$. Therefore, to do the energy estimates in the interior, we still need to establish the bound of the energy flux of $F$ on the null cone.

First we note that $|\rho[F]-\rho[\tilde F]|\lesssim |\q_0| r^{-2}$, and other components are the same. When we use the multiplier $S$, the energy flux of $F$ on the cone $C=\{t-r=1\}\cap\{t\geq 2\}$ is like
\begin{equation*}
    \int_{C} \rho[F]^2+\sigma[F]^2+r|\a[F]|^2 dvol,
\end{equation*}
so the difference, i.e. the integral with $F$ replaced by $F-\tilde F$, is integrable. Now we need to estimate the difference with vector fields, and we will show that the difference is also integrable for any number of vector fields.

Denote $\tau_+=t+r$, $\tau_-=t-r$. We now compute various commutators of vector fields. First we have
\begin{equation*}
    [\Om_{0i},L]=\frac tr\pab_i-\pa_i-\om_i\pa_t=-\om_i L+\frac{\tau_-}{r}\pab_i,\quad [\Om_{ij},L]=0.
\end{equation*}
In order to apply more Lie derivatives on $L$ and $\Lb$, we need the identities
\begin{equation*}
    [\Om_{0i},\pab_j]=-r^{-1}(\delta_{ij}-\om_i\om_j)\om^k \Om_{0k}, \quad [\Om_{ij},\pab_k]=-(\delta_{ik}-\om_i\om_k)\pab_j+(\delta_{jk}-\om_j \om_k)\pab_i.
\end{equation*}
We also have $\Om^I(\om_i)=O(1)$, $\Om^I (r)=O(r)$ when $t\approx r$ for any muli-index $I$. Therefore
\begin{equation*}
    \Lie_\Om^I \pab_i=O(r^{-1})\Om+O(1)\slashed\pa+O(r^{-1})\pa=O(1)L+O(r^{-1}\langle\tau_-\rangle)\Lb+O(1)\slashed\pa.
\end{equation*}
We also have $\Om^I (\tau_-)=O(\tau_-)$. Therefore using the $[\Om,L]$ formulas we have
\begin{equation*}
    \Lie_\Om^I L=O(1) L+O(r^{-1}\tau_-)\slashed\pa+O(r^{-2}\tau_-^2)\Lb.
\end{equation*}
For partial derivatives, we have $[\pa_t,L]=0$ and $[\pa_i,L]=r^{-1}\pab_i$. Therefore we have
\begin{equation*}
    \Lie_Z^I L=O(1) L+O(r^{-1}\tau_-)\slashed\pa+O(r^{-2}\tau_-^2)\Lb.
\end{equation*}

Similarly we have
\begin{equation*}
    [\Om_{0i},\Lb]=-\om_i\Lb+\frac{\tau_+}r \pab_i,\quad [\Om_{ij},\Lb]=0
\end{equation*}
which yields
\begin{equation*}
    \Lie_Z^I \Lb=O(1)\Lb+O(1)(\slashed\pa,L).
\end{equation*}

We shall make the following induction assumptions
\begin{equation*}
    |(\a,\sigma)[\Lie_Z^{I}(F-\tilde F)]|\lesssim |\q_0|\langle\tau_-\rangle r^{-3}, \quad |\rho[\Lie_Z^{I}(F-\tilde F)]|\lesssim |\q_0| r^{-2},\quad |\underline\a[\Lie_Z^{I}(F-\tilde F)]|\lesssim |\q_0|r^{-2}
\end{equation*}
for $|I|\leq k-1$. Denote $G=F-\tilde F$. For $i=1,2,3$, we have
\begin{equation*}
    \a[\Lie_Z^I G](\pab_i)=\Lie_Z^I G_{L\pab_i}=Z^I (G_{L\pab_i})-\sum_{\substack{|I_1|+|I_2|+|I_3|=|I|\\|I_1|<|I|}} (\Lie_Z^{I_1} G)(\Lie_Z^{I_2} L,\Lie_Z^{I_3} \pab_i).
\end{equation*}
Using the estimates above and the induction assumptions, we have for $|I|\leq k-1$ ($\slashed\pa=\pab_i$, $i=1,2,3$)
\begin{equation*}
\begin{split}
    |\a[\Lie_Z^I G]|&\lesssim \sum_{|J|\leq k-1}\tau_- r^{-1}|\Lie_Z^J G(L,\Lb)|+|\Lie_Z^J G(L,\slashed\pa)|+r^{-2}\tau_-^2 |\Lie_Z^J G(\Lb,\slashed\pa)|+r^{-1}\tau_- |\Lie_Z^J G(\slashed\pa,\slashed\pa)|\\
    &\lesssim |\q_0| \tau_- r^{-3}.
    \end{split}
\end{equation*}
Similarly,
\begin{equation*}
\begin{split}
    |\rho[\Lie_Z^I G]|&\lesssim |Z^I (G_{L\Lb})|+\sum_{|J|\leq k-1} |\Lie_Z^J G\left(O(1)(L,\slashed\pa)+O(r^{-1}\tau_-)\Lb,O(1)(\Lb,L,\slashed\pa)\right)|\\
    &\lesssim\sum_{|J|\leq k-1}|\Lie_Z^J G(L,\Lb)|+|\Lie_Z^J G(L,\slashed\pa)|+|\Lie_Z^J G(\Lb,\slashed\pa)|+|\Lie_Z^J G(\slashed\pa,\slashed\pa)|\\
    &\lesssim |\q_0| r^{-2},
    \end{split}
\end{equation*}
\begin{equation*}
\begin{split}
    |\sigma[\Lie_Z^I G]|&\lesssim \sum_{|J|\leq k-1} |\Lie_Z^J G\left(O(1)L+O(r^{-1}\langle\tau_-\rangle)\Lb+O(1)\slashed\pa,O(1)L+O(r^{-1}\langle\tau_-\rangle)\Lb+O(1)\slashed\pa\right)|\\
    &\lesssim\sum_{|J|\leq k-1}r^{-1}\langle\tau_-\rangle|\Lie_Z^J G(L,\Lb)|+|\Lie_Z^J G(L,\slashed\pa)|+r^{-1}\langle\tau_-\rangle|\Lie_Z^J G(\Lb,\slashed\pa)|+|\Lie_Z^J G(\slashed\pa,\slashed\pa)|\\
    &\lesssim |\q_0| \langle\tau_-\rangle r^{-3},
    \end{split}
\end{equation*}
\begin{equation*}
\begin{split}
    |\underline\a[\Lie_Z^I G]|&\lesssim \sum_{|J|\leq k-1} |\Lie_Z^J G\left(O(1)(\Lb,L,\slashed\pa),O(1)L+O(r^{-1}\langle\tau_-\rangle)\Lb+O(1)\slashed\pa\right)|\\
    &\lesssim\sum_{|J|\leq k-1}|\Lie_Z^J G(L,\Lb)|+|\Lie_Z^J G(L,\slashed\pa)|+|\Lie_Z^J G(\Lb,\slashed\pa)|+|\Lie_Z^J G(\slashed\pa,\slashed\pa)|\\
    &\lesssim |\q_0| r^{-2}.
    \end{split}
\end{equation*}
Therefore we prove the estimates for $k$, hence the estimates for arbitrary number of vector fields. As a result, we obtain the same bounds of the null cone flux for $F$.

\subsection{Bootstrap arguments}
In this subsection, we briefly review the bootstrap argument which establishes the existence in the interior. As is mentioned above, with the slight generalization to higher orders in previous section, the argument is now essentially the same (in fact a bit easier) as in \cite{FangWangYang}. We make no claim of originality on this part.

First we recall the exterior result \cite{KlainermanWangYang}. While the original result is stated with no greater than two order vector fields, it was pointed out, and can be also seen from above, that one can easily generalize it to higher orders. The initial value problem in gauge invariant form is formulated as follows: We consider the electric part $E_i=F_{0i}$ and magnetic part $H_i=\prescript{\star}{}{F_{0i}}$, where $\star$ is the Hodge dual operator. The initial data is the set $(E,H,\phi,\dot\phi_0)$ which prescribes the value of $E$, $H$, $\phi$ and $D_0\phi$, but the constraint equations need to be satisfied:
\begin{equation*}
    \div E=\Im(\phi_0\, \overline{\dot\phi_0}),\quad \div H=0.
\end{equation*}
We define the chargeless part of $E$ by $\tilde E_i=E_i-\q_0\chi_{ex}(r)r^{-1}\om_i$.
We now state the result, mainly the part that we need:

\begin{thm}[Exterior solution, \cite{KlainermanWangYang}]\label{exteriorthm}
Consider the Cauchy problem for (\ref{mMKG}) with the admissible initial data set $(E,H,\phi_0,\dot\phi_0)$ (satisfying constraint equations) at $\{t=2\}$. Then if for small $\varepsilon>0$ and some $1<\gamma_0<2$,
\begin{equation}
\label{ExteriorCondition}
\sum\limits_{|I| \leq N}\int_{r\geq 1}(1+r)^{\gamma_0+2|I|}(|\bar D\bar D^I\phi_0|^2+|\bar D^I\dot\phi_0|^2+|\bar D^I\phi_0|^2+|\bar\pa^I \tilde{E}|^2+|\bar\pa^I H|^2)dx\leq \varepsilon,
\end{equation}  
where $\bar D$ and $\bar \pa$ are projections of $D$ and $\pa$ to $\mathbb{R}^3$ respectively, then the unique local solution $(F, \phi)$ of the mMKG system can be extended to the whole exterior region $\{t-r\leq 1\}$ with the energy bounds
\begin{equation*}
q_+^{\gamma_0-1}\int_{C_q}r|D_L D_Z^I\phi|^2+|\slashed DD_Z^I\phi|^2+|D_Z^I\phi|^2+
\int_{C_q} r^{\gamma_0}|\a[\Lie_Z^I \tilde{F}]|^2+q_+^{\gamma_0} \int_{C_q} |\rho[\Lie_Z^I \tilde F|^2+|\sigma[\Lie_Z^I F]|^2\leq C\varepsilon
\end{equation*}
for $k\leq N$, where $C_q=\{r-t=q\}$ are outgoing null hypersurfaces. We also have the decay estimates of the scalar field
$|D_Z^I\phi| \leq C \varepsilon r^{-\frac 32}q_+^{-\frac{\gamma_0}2}$ 
for $k\leq N-2$.
\end{thm}

\vspace{2ex}

Here we let $N\geq 5$. This in particular implies the bound of the energy flux on the cone $\{t-r=1\}$, thanks to the estimates in last subsection:
\begin{equation}
    C_k[F,\phi]=\int_{t-r=1} |(\rho,\sigma)[\Lie_Z^I F]|^2+r|\a[\Lie_Z^I F]|^2+|D_L D_Z^I\phi|^2+|\slashed D D_Z^I\phi|^2+|D_Z^I\phi|^2 \leq C\varepsilon.
\end{equation}
Now we use $E^H_{S}(F,\tau)$ and $E^H_{KG}(\phi,\tau)$ to denote the energy of the Maxwell field and the scalar field on truncated hyperboloids $\hat H_\tau$, induced by the multiplier $S$ and $\pa_t$ respectively. We also define $E^H_{S,k}=\sum_{|I|\leq k}E^H_S(\Lie_Z^I F,\tau)$, $E^H_{KG,k}=\sum_{|I|\leq k}E^H_{KG}(D_Z^I\phi,\tau)$. We have the comparability
\begin{equation*}
    E^H_{S}(\tau)\approx \sum_{|I|\leq k}\int_{\hat H_\tau} \frac{\tau^2}t|\Lie_Z^I F^H|^2 dx,\quad E_{KG}^H(\tau)\approx \sum_{|I|\leq k}\int_{\hat H_\tau} (\frac \tau t)^2 |D_{\pa_t}\phi|^2+(\frac \tau t)^2|D_R\phi|^2+|\slashed D\phi|^2+|\phi|^2 dx,
\end{equation*}
where $R=\tau^{-1}\Om_{0r}$.
We set up the following bootstrap assumptions:
\begin{equation*}
    E_{S,k}^H(\tau)^\frac 12\leq C_b \varepsilon,\quad E_{KG,k}^H(\tau)^\frac 12\leq C_b \varepsilon\, \tau^\delta,\quad k\leq N.
\end{equation*}
We do not include $C_b$ in the implicit constants from ``$\lesssim$".


Using the Klainerman-Sobolev type estimate on hyperboloids, we immediately get the decay estimate of the scalar field:
\begin{equation*}
    \frac{\tau}t|DD_Z^I\phi|+|D_Z^I\phi|\lesssim C_b\varepsilon t^{-\frac 32}\tau^\delta,\quad |I|\leq N-2.
\end{equation*}
To obtain the decay of components of the Maxwell field one can also decompose the tensor to electric and magnetic parts under the hyperboloidal frame as in \cite{FangWangYang} and then apply the Sobolev estimate to get
\begin{equation}
    |(\Lie_Z^I F)^H|\lesssim C_b\varepsilon t^{-1}\tau^{-1}, \quad |I|\leq N-2.
\end{equation}

\subsubsection{The Maxwell field}
For the Maxwell field, we have no commutator terms, so we have
\begin{equation*}
    \begin{split}
        E^H_{S,k}(\tau)\lesssim E^H_{S,k}(\tau_0)+C_k(\tau)+\sum_{|I|\leq k}\left|\iint_{D_{\tau_0}^\tau} S^\mu \Lie_Z^I F_{\mu\nu} \Lie_Z^I J^\nu[\phi]\right|
    \end{split}
\end{equation*}
where $D_{\tau_0}^\tau$ is the region enclosed by $\hat H_{\tau_0}$, $\hat H_\tau$, and part of the cone $\{t-r=1\}$ between the two hyperboloids, denoted by $C_{\tau_0}^\tau$. The last integral is bounded by the sum of terms like ($|I|=k$)
\begin{equation*}
    \sum_{|I_1|+|I_2|\leq |I|}\left|\iint_{D_{\tau_0}^\tau} S^\mu \Lie_Z^{I} F_{\mu\nu} \Im (D_Z^{I_1}\phi\c\overline{D^\nu D_Z^{I_2}\phi})\right|+\sum_{|I_3|+|I_4|+|I_5|<|I|} \left|\iint_{D_{\tau_0}^\tau} S^\mu \Lie_Z^{I}\Im(D_Z^{I_3}\phi\c\overline{i\Lie_Z^{I_4} F_{Z\mu} D_Z^{I_5}\phi})\right|
\end{equation*}
The second here contains an additional decay factor so we omit it. For the first term, when $|I_2|<|I|$, we can use the following estimate
\begin{equation*}
    \left|\iint_{D_{\tau_0}^\tau} S^\mu \Lie_Z^{I} F_{\mu\nu} \Im (D_Z^{I_1}\phi\c\overline{D^\nu D_Z^{I_2}\phi})\right|\leq \iint_{D_{\tau_0}^\tau} \tau |\Lie_Z^I F^H| |D_Z^{I_1}\phi|\c\tau^{-1}|D_\Om D_Z^{I_2}\phi|
\end{equation*}
since one component of $\Lie_Z^I F$ is $S$. Here we used the relation $|\slashed D\phi|+|D_R\phi|\lesssim \tau^{-1}|D_{\Om_{0i}}\phi|$, which can be seen from the relation $t\slashed\pa_i=\Om_{0i}-\om_i\Om_{0r}$. Note that here the volume element is $dvol=\frac{\tau}t dx d\tau$. The bootstrap assumption says
\begin{equation*}
    \int_{\hat H_\tau} \frac{\tau^2}t|(\Lie_Z^I F)^H|^2 dx\lesssim (C_b\varepsilon)^2,\quad |I|\leq k.
\end{equation*}
Therefore the integral can be controlled by
\begin{equation*}
    \sum_{|I_1|+|I_2|\leq |I|} \int_{\tau_0}^\tau \left(\int_{\hat H_\tau} \frac{\tau^2}t |\Lie_Z^I F^H|^2 dx\right)^\frac 12 \left(\int_{\hat H_\tau} \tau^{-1} |D_Z^{I_1}\phi|^2 |D_\Om D_Z^{I_2}\phi|^2 dx\right)^\frac 12 d\tau
\end{equation*}
One of the scalar field factor can be estimated by the decay, so e.g. we get
\begin{equation*}
    \int_{\hat H_\tau} C_b\varepsilon t^{-\frac 32}\tau^{-\frac 12+\delta} \left(\int_{\hat H_\tau} \frac{\tau^2}t |\Lie_Z^I F^H|^2 dx\right)^\frac 12 \left(\int_{\hat H_\tau}  |D_\Om D_Z^{I_2}\phi|^2 dx\right)^\frac 12\leq C_b^3 \varepsilon^3.
\end{equation*}

When $|I_2|=|I|$, the integration by part argument in \cite{FangWangYang} plays a fundamental role. Since $\pa^\mu S^\nu=m^{\mu\nu}$, we have
\begin{equation*}
    \begin{split}
        S^\nu \Lie_Z^I F_{\mu\nu} \Im &(\phi\c\overline{D^\mu D_Z^I\phi})\\
        &=\pa^\mu (S^\nu \Lie_Z^I F_{\mu\nu} \Im(\phi\c\overline{D_Z^I \phi}))-S^\nu (\pa^\mu \Lie_Z^I F_{\mu\nu}) \Im(\phi\c\overline{D_Z^I\phi})-S^\nu \Lie_Z^I F_{\mu\nu} \Im(D^\mu\phi\c\overline{D_Z^I \phi}).
    \end{split}
\end{equation*}
Notice that $\pa^\mu\Lie_Z^I F_{\mu\nu}=-\Lie_Z^I J[\phi]_\nu$. This is quadratic, so we obtain an extra decay power. The last term can be treated similarly as above. Therefore it suffices to control the boundary flux
\begin{equation*}
    \left|\int_{\hat H_\tau} \tau^{-1} S^\mu S^\nu \Lie_Z^I F_{\mu\nu} \Im(\phi\c\overline{D_Z^I\phi})\right|+\left|\int_{C_{\tau_0}^\tau} L^\mu S^\nu \Lie_Z^I F_{\mu\nu} \Im(\phi\c\overline{D_Z^I\phi})\right|+\left|\int_{\hat H_{\tau_0}} \tau_0^{-1} S^\mu S^\nu \Lie_Z^I F_{\mu\nu} \Im(\phi\c\overline{D_Z^I\phi})\right|.
\end{equation*}
Using the anti-symmetry of $\Lie_Z^I F$, we see that the first and last term are zero. Also, using $S=\frac 12(\tau_+L+\tau_-\Lb)$, the decay (boundedness) of $\phi$, and the bounds on the cone flux we get
\begin{equation*}
    |\int_{C_{\tau_0}^\tau} L^\mu S^\nu \Lie_Z^I F_{\mu\nu} \Im(\phi\c\overline{D_Z^I\phi})|\lesssim |\int_{C_{\tau_0}^\tau} \rho[\Lie_Z^I F] \Im(\phi\c\overline{D_Z^I\phi})|\lesssim C_b^3\varepsilon^3,
\end{equation*}
Therefore we get $E_{S,k}^H(\tau)\lesssim \varepsilon^2+C_b^3 \varepsilon^3$, so making $C_b$ relatively big (compared with the implicit constants) and then $\varepsilon$ small we improve the boostrap assumption on $E_{S,k}^H(\tau)$.

\subsubsection{The scalar field}
For the scalar field, the right hand sides are commutators, which leads to the integral in the estimate of $E_{KG,k}^H(\tau)$:
\begin{equation*}
    \sum_{|I_1|+|I_2|<|I|}\iint_{D_{\tau_0}^\tau} |D_0 D_Z^I\phi||Q(\Lie_Z^{I_1}F,D_Z^{I_2}\phi,Z)|+\sum_{|I_3|+|I_4|+|I_5|<|I|}\iint_{D_{\tau_0}^\tau} |D_0 D_Z^I\phi||\Lie_Z^{I_3} F_{Z\mu} \Lie_Z^{I_4} F^\mu_{\, \, Z} D_Z^{I_5}\phi|.
\end{equation*}
The second term behaves better so we omit it.
Since $\Om_{0i}=\om_i\Om_{0r}+t\pab_i$, the main contribution of the trilinear form $Q$ can be estimated as
\begin{equation*}
    |Z^\nu \Lie_Z^{I_1} F_{\mu\nu} D^\mu D_Z^{I_2}\phi|\lesssim t|\Lie_Z^{I_1} F^H|(|\slashed DD_Z^{I_2}\phi|+|D_{R}D_Z^{I_2}\phi|+|D_{\pa_\rho}D_Z^{I_2}\phi|)
\end{equation*}
We have
\begin{equation*}
    |\slashed D\phi|+|D_R\phi|\lesssim \tau^{-1}|D_{\Om_{0i}}\phi|,\quad |D_{\pa_\tau}\phi|\lesssim \frac \tau t |D_{\pa_t}\phi|+\tau^{-1}|D_{\Om_{0i}}\phi|,
\end{equation*}
so the main term becomes $\tau |\Lie_Z^{I_1} F^H| |D_{\pa_t}\phi|$. We can then estimate
\begin{equation*}
\begin{split}
    \iint_{D_{\tau_0}^\tau} |D_0 D_Z^I\phi|&|Q(\Lie_Z^{I_1}F,D_Z^{I_2}\phi,Z)|\frac\tau t dx d\tau \leq \iint_{D_{\tau_0}^\tau} \frac \tau t \c\tau|D_0 D_Z^I \phi||\Lie_Z^{I_1} F^H||D D_Z^{I_2}\phi| dx d\tau\\
    &\leq \int_{\tau_0}^\tau s\left(\int_{\hat H_s} (\frac s t |D_0 D_Z^I \phi|)^2 dx\right)^\frac 12 \left(\int_{\hat H_s} |\Lie_Z^{I_1} F^H|^2 |D D_Z^{I_2}\phi|^2 dx\right)^\frac 12 ds\\
    &\leq \int_{\tau_0}^\tau s^{1+\delta}\c C_b\varepsilon \left(\int_{\hat H_s} C_b^2 \varepsilon^2 t^{-2} \tau^{-2} |D D_Z^{I_2}\phi|^2+(C_b\varepsilon t^{-\frac 12}\tau^{-1+\delta})^2 |\Lie_Z^{I_1} F^H|^2 dx\right)^\frac 12 ds\\
    &\leq \int_{\tau_0}^\tau C_b^2\varepsilon^2 s^{-1+\delta}\left(\int_{\hat H_s} t^{-2} s^{2} |D D_Z^{I_2}\phi|^2 dx+\int_{\hat H_s} s^{2\delta} \frac{s^2}{t}|\Lie_Z^{I_1} F^H|^2 dx \right)^\frac 12 ds\\
    &\leq \int_{\tau_0}^\tau C_b^3 \varepsilon^3 s^{-1+2\delta}\, ds \lesssim C_b^3 \varepsilon^3 \tau^{2\delta}.
\end{split}
\end{equation*}
Again, letting $C_b$ be relatively big and then $\varepsilon$ be small we can improve the bootstrap assumption on $E_{k,KG}^H(\tau)$. Therefore we have improved both assumptions and hence obtain the global existence, with global decay estimates.

\section{Asymptotics}
\subsection{Initial data}
Under the Lorenz gauge, the initial data needs to satisfy the gauge condition as well as the constraint equation, which comes from the Maxwell equation:
\begin{equation*}
    \pa^\mu A_\mu=0,\quad -\triangle A_0=J_0-\pa^i(\pa_t A_i),
\end{equation*}
where $J_0=\Im(\phi\, \overline{D_0\phi})=\Im(\phi\, \overline{\pa_t\phi})-|\phi|^2 A_0$.
Therefore the initial data set can be set to be $(A_i,\pa_t A_i,\phi,\pa_t\phi)|_{t=2}=(a_i,\dot{a}_i,\phi_0,\dot\phi_0)$ with $i=1,2,3$. The remaining components $(a_0,\dot a_0):=(A_0,\pa_t A_0)|_{t=2}$ can then be determined by the conditions above.

Define the norm
\begin{equation*}
    ||\psi||^2_{H^{k,\gamma_0}}:=\sum_{|I|\leq k} \int_{\mathbb{R}^3} (1+r)^{\gamma_0+2|I|} |\pa_x^I \psi|^2 dx.
\end{equation*}
We impose the initial data set $(a_i,\dot a_i,\phi_0,\dot\phi_0)$ satisfying the bound with $\gamma_0>1$:
\begin{equation}
    ||a_i||_{H^{k+1,\gamma_0-2}}+||\dot a_i||_{H^{k,\gamma_0}}+||(\phi_0,\pa_x \phi_0,\phi_1)||_{H^{k,\gamma_0}}\leq \varepsilon.
\end{equation}

We want to verify the initial energy condition \eqref{ExteriorCondition} to obtain the existence. This requires also the information of $A_0$ initially. Since $\q_0=\frac 1{4\pi}\int_{\mathbb{R}^3} J_0$ is generally nonzero, we expect a $r^{-1}$ tail of $A_0$ since it satisfies an elliptic equation. Define
\begin{equation*}
    \tilde a_0=a_0-\q_0 \chi_{ex}(r) r^{-1},
\end{equation*}
where $\chi_{ex}$ is a non-decreasing smooth cutoff function with $\chi_{ex}(s)=0$ when $s\leq -1$, and $\chi_{ex}(s)=1$ when $s\geq -1/2$.
Then at the initial slice $E_i=F_{0i}=\pa_t a_i-\pa_i a_0=\pa_t a_i-\pa_i \tilde a_0+\q_0 \chi_{ex}(r) r^{-2}\om_i-\q_0 \chi_{ex}'(r)\om_i r^{-1}$. The last term vanishes when $r\geq 1$. Therefore $\tilde E_i=E_i-\q_0 \chi_{ex}(r) r^{-2} \om_i$ coincides with the previous definition when $r\geq 1$. (To construct the exterior solution, we only need to care about the part where $r\geq 1$.)

One can establish the estimate of $\tilde a_0$ using elliptic estimates, see details in \cite{CKL}:
\begin{lem}
    If $||a_i||_{H^{k+1,\gamma_0-2}}+||\dot a_i||_{H^{k,\gamma_0}}+||(\phi_0,\pa_x \phi_0,\dot\phi_0)||_{H^{k,\gamma_0}}\leq \varepsilon$ for some $\gamma_0\in (1,3)$, then for any $\gamma_0'<\gamma_0$ we have
    \begin{equation*}
        ||a_0-\q_0 \chi_{ex}(r) r^{-1}||_{H^{k+1,\gamma_0'-2}}+||\dot a_0||_{H^{k,\gamma_0'}}\lesssim \varepsilon.
    \end{equation*}
\end{lem}
To verify the condition of the exterior theorem (with $\gamma_0$ replaced by $\gamma_0'$), we need to check the norm of $\tilde E$, $H$, $\phi_0$ and $\dot\phi_0$ on $\{t=2,r\geq 1\}$. When $r\geq 1$ we have $\tilde E_i=\dot a_i-\pa_i \tilde a_0$, so we can obtain the bound of $\tilde E_i$ using what we have for $\tilde a_0$. For the scalar field it suffices to replace partial derivatives by the covariant derivatives, which is because of the bounds of $A_\mu$ and Sobolev estimates (this is fairly easy and we leave it to the next subsection). Therefore the data satisfies the initial condition and we have a global solution with estimates in the previous section.

\subsection{Decay estimates}
Recall that under the Lorenz gauge we have
\begin{equation*}
    -\Box A_\mu=J[\phi]_\mu=\Im (\phi\overline{D_\mu\phi}).
\end{equation*}
Similar to above we consider the part with the charge $\q_0 \chi_{ex}(r-t) r^{-1}\delta_{0\mu}$. Notice that this is an exact solution to the homogeneous wave equation. Therefore we have for $\tilde A_\mu=A_\mu-\q_0 \chi_{ex}(r-t) r^{-1}\delta_{0\mu}$ that 
\begin{equation*}
    -\Box \tilde A_\mu=\Im(\phi\overline{D_\mu \phi})
\end{equation*}
with initial data $(a_0,\dot a_0)$.
From the existence part we have $|D_Z^I\phi|\lesssim \varepsilon(1+t+r)^{-\frac 32}\langle \tau\rangle^\delta q_+^{-\frac{\gamma_0'}2}$ (set $\tau=1$ in the exterior), so in particular we have
\begin{equation*}
    |\phi|+|D\phi|\lesssim \varepsilon t^{-\frac 32}\tau^{2\delta} q_+^{-\frac{\gamma_0'}2}\lesssim \varepsilon t^{-3+2\delta}\mathbf{1}_{t\geq r-1}+\varepsilon r^{-\frac 32}  q_+^{-{\gamma_0'}}\mathbf{1}_{t<r-1}.
\end{equation*}
Here $\mathbf 1$ is the indicator function.
This gives an estimate of the right hand side of the wave equations. By linearity, we can analyze the effect of the inhomogeneous term and the initial data separately. For the inhomogeneous terms above, if we consider a wave equations with them on the right hand side with vanishing initial data, then the solution (denoted by $\tilde A_\mu^{source}$) has the decay
\begin{equation}
    |\tilde A_\mu^{source}|\lesssim \varepsilon^2 (1+t+r)^{-1+\delta} (1+q_+)^{-{\gamma_0'}}.
\end{equation}
The estimate is standard, and here one can treat the interior and exterior sources separately. See \cite[Appendix]{CKL} for the radial estimate method and \cite{LeFlochMamodel} for alternative proof using the representation formula.


For the initial data, from the initial bound of $A_\mu$ and the following Sobolev estimate we can derive the initial decay.
\begin{lem}
    Let $\psi$ be a function on $\mathbb{R}^3$. We have
    \begin{equation*}
        \sup_x\, \langle r\rangle^{\frac{\gamma_0'+1}2}|\psi|\leq C\sum_{|\b|+k\leq 3} \left(\int_{\mathbb{R}^3} \langle r\rangle^{\gamma_0'-2} |(\langle r\rangle \pa_r)^k \pa_\om^\b \psi|^2 dx\right)^\frac 12.
    \end{equation*}
\end{lem}
Then applying this to the zeroth and first order derivative of $\tilde A_\mu$ at the initial slice, we get $|\tilde A_\mu||_{t=2}\lesssim \varepsilon (1+r)^{-\frac{\gamma_0'+1}2}$, and $|\pa \tilde A_\mu||_{t=2}\lesssim \varepsilon (1+r)^{-\frac{\gamma_0'+3}2}$. Now we need the following estimate, which is also standard:
\begin{lem}
    If $w$ is the solution of the homogeneous wave equation $-\Box w=0$ with $(w,\pa_t w)|_{t=0}=(w_0,w_1)$, then for any $\alpha\in (0,1)$ we have
    \begin{equation}
        (1+t+r)(1+|t-r|)^\a |w(t,x)|\leq C\sup_x \left((1+|x|)^{2+\a} (|w_1(x)|+|\pa w_0(x)|)+(1+|x|)^{1+\a} |w_0(x)|\right).
    \end{equation}
\end{lem}
Therefore, the contribution from the initial data is of the size $\varepsilon (1+t+r)^{-1} (1+|t-r|)^{-\frac{\gamma_0'-1}2}$. So by linearity we get the estimate 
\begin{equation}
    |\tilde A_\mu|\lesssim \varepsilon t^{-1+2\delta}\mathbf{1}_{t\geq r}+\varepsilon(1+t+r)^{-1}(1+q_+)^{-\frac{\gamma_0'-1}2}\mathbf{1}_{t<r}.
\end{equation}
\begin{remark}
    In the statement of Theorem \ref{mainthm1}, we extract the linear part $A^{Linear}_\mu$, which is exactly the initial data part we discussed here. Due to the slow decay of the initial data, this term decays less than other lower order terms in the interior. Nevertheless, it does not affect the main behavior in the interior either.
\end{remark}
It is also straightforward to show this with vector fields applied to $\tilde A_\mu$, i.e. $Z^I \tilde A_\mu$, by commuting the wave equation of $\tilde A_\mu$ with $Z^I$. This serves as a preliminary bound of $\tilde A_\mu$. 

Later we will also use improved bounds of the derivative of $\tilde A_\mu$, which is expected to be better than the decay of themselves. We already have better behaviors for derivatives of the initial data part (because in this case we can easily commute the linear wave equation with $(Z,S)$). For the source part, we decompose the source $$\Im(\phi \overline{D_\mu\phi})=\chi_{ex}(q)\Im(\phi \overline{D_\mu\phi})+(1-\chi_{ex}(q))\Im(\phi \overline{D_\mu\phi}).$$
We first note that the size is well-preserved when $Z$ or $S$ fall on the cutoff because $|(Z,S)q|\lesssim 1$.
For the exterior part, we can commute the wave equation with $S$: Using $[\Box,S]=2\Box$, we have $|\Box \tilde A_\mu^{source}|\lesssim |S(\chi_{ex}(q)\Im(\phi \overline{D_\mu\phi}))|+|\chi_{ex}(q)\Im(\phi D_\mu\phi)|$. Then since $S=\Om_{0r}+(t-r)\Lb$, we have $|S(\chi_{ex}(q)\Im(\phi\overline{D_\mu\phi}))|\lesssim \varepsilon^2 t^{-3+2\delta}(1+|t-r|)(1+q_+)^{-\gamma_0'}$. We see that the exterior source is still decaying enough, so this part gives a solution $S\tilde A_\mu^{source,ex}$ satisfying $\tau|\pa_\tau A_\mu^{source,ex}|=|SA_\mu^{source,ex}|\lesssim \varepsilon^2\langle t+r\rangle^{-1+2\delta}$. 

For the interior source $(1-\chi_{ex}(q))\Im(\phi \overline{D_\mu\phi})$, this is now a problem supported in $\{t-r\geq 1\}$. We can simply use the conformal energy estimate on the hyperboloids as in \cite{CL} which deals with compactly supported data:
\begin{equation*}
    E_{con}(Z^I \tilde A_\mu^{source,in},\tau)[H_\tau]^\frac 12\lesssim E_{con}(Z^I \tilde A_\mu^{source,in},\tau_0)[H_{\tau_0}]^\frac 12+\int_{\tau_0}^\tau \varepsilon s^{-\frac 12+\delta}\sum_{|J|\leq |I|+2}||D_Z^J \phi||_{L^2(H_s)}\lesssim \varepsilon \tau^{\frac 12+2\delta},
\end{equation*}
The norm $||\tau^2 t^{-1} \pa_\tau Z^I\tilde A_\mu^{source,in}||_{L^2(H_\tau)}$ is bounded by the conformal energy.
Then by the Klainerman-Sobolev estimate on hyperboloids, this implies $|\tilde A_\mu^{source,in}|\lesssim \varepsilon^2 t^{-\frac 12}\tau^{-\frac 32+2\delta}$. Therefore we have $|\pa_\tau\tilde A_\mu^{source}|\lesssim \varepsilon t^{-\frac 32}\tau^{-\frac 12+2\delta}$. Combining these two source parts with the linear part we can get $|\pa_\tau \tilde A_\mu|\lesssim \varepsilon t^{-\frac 12}\tau^{-\frac 32+2\delta}$.


\vspace{1ex}

We also need a weak decay of derivatives of the scalar field without the connection. For example, we have
$|Z_1 Z_2 \phi|\leq |D_{Z_1} D_{Z_2} \phi|+|Z_1^\mu A_\mu Z_2\phi|+|Z_1(Z_2^\nu A_\nu)\phi|+|(Z_1^\mu A_\mu)(Z_2^\nu A_\nu)\phi|\lesssim \varepsilon\langle t+r\rangle^{-\frac 32+3\delta}$. Generally we have $|Z^I \phi|\lesssim \varepsilon \langle t+r\rangle^{-\frac 32+(2|I|+1)\delta}$. Note that $\delta$ is small, and we are not doing boostrap arguments in this section, so we will not be very careful on the coefficients in front of $\delta$.

\vspace{3ex}

To derive the asymptotic behaviors of the Klein-Gordon field, we will integrate along the hyperboloidal rays ($r/t$ being constants less than $1$). This requires a ``initial condition" provided by the initial data and the decay of the field near the light cone. The decay of the scalar field $\phi$ on the cone, from the exterior result, are
\begin{equation*}
    |D_Z^I\phi|\lesssim \varepsilon r^{-\frac 32}.
\end{equation*}
We remark here this decay on the cone can be improved. In \cite{LeFLochMa2022model}, LeFloch-Ma observed the following identity:
\begin{equation*}
    -\pa_t^2\phi+\triangle_x\phi=\frac{r^2-t^2}{t^2}\pa_t^2\phi-\frac{2x^i}{t^2}\pa_t \Om_{0i}\phi+t^{-2}\Om_{0i}\Om_{0i}\phi+\frac{x^i}{t^3}\Om_{0i}\phi-t^{-1}(3+\frac{r^2}{t^2})\pa_t\phi
\end{equation*}
where we take the sum on repeated spatial indices. This yields the estimate
\begin{equation*}
    |\Box\phi|\lesssim \varepsilon t^{-1}(1+|t-r|)\sum_{|I|\leq 2} |Z^I\phi|.
\end{equation*}
Therefore, using the bounds we get and the equation
\begin{equation*}
    \phi=\Box\phi+2iA^\mu \pa_\mu\phi-A^\mu A_\mu\phi,
\end{equation*}
we obtain
\begin{equation*}
    |\phi|\lesssim \varepsilon t^{-\frac 52+\delta}(1+(t-r))\lesssim \varepsilon\tau^{-\frac 32+\delta} (1-|y|^2)^{\frac 74-\frac\delta 2}+\varepsilon\tau^{-\frac 52+\delta} (1-|y|^2)^{\frac 54-\frac \delta 2} , \quad\text{when }t>r.
\end{equation*}
So in particular $\phi$ decays at the rate $r^{-\frac 52+\delta}$ along the light cone. Similar bounds hold with some vector fields.

\subsection{Asymptotics}
\subsubsection{The scalar field}
Recall that the scalar field equation can be written under the Lorenz gauge as
\begin{equation*}
    -\Box \phi+\phi=2iA^\mu \pa_\mu \phi-A^\mu A_\mu \phi.
\end{equation*}
The charge part of $A_\mu$, i.e. $\q_0 \chi_{ex}(r-t) r^{-1}\delta_{0\mu}$, is supported around the light cone and contribute little to our analysis below. Therefore we focus on the chargeless part and with a slight abuse of notation denote $\tilde A_\mu$ by $A_\mu$ in this subsection. 

We decompose $A^\mu \pa_\mu\phi$ in the hyperboloidal frame:
\begin{equation}
    A^\mu \pa_\mu\phi=(A^\tau \pa_\tau \phi+R_{tan,1})\chi_{LC}(q)+(1-\chi_{LC}(q))A^\mu \pa_\mu\phi.
\end{equation} 
The cutoff function $\chi_{LC}(q)$ is $1$ when $q\leq -1$ and $0$ when $q\geq -\frac 12$. We introduce this cutoff because the decomposition in the hyperboloidal frame becomes irregular close to the the light cone. The term $R_{tan,1}$ involves derivative tangential to hyperboloids, and is of the size $O(t^{-1}|A||\nab_y \phi|)$. Then, in the support of $\chi_{LC}(q)$, we have
\begin{equation*}
    \begin{split}
        |R_{tan,1}|&\lesssim \varepsilon t^{-1} t^{-1+\delta} (1-|y|^2)^{-1} |\Om\phi|\lesssim \varepsilon^2 t^{-2+\delta} (1-|y|^2)^{-1} t^{-\frac 52+\delta}(t-r)\\
        &\lesssim \varepsilon^2\tau^{-\frac 72+2\delta} (1-|y|^2)^{\frac 74-\delta}
    \end{split}
\end{equation*}
which is good. We also denote $R_{LC,1}=(1-\chi_{LC}(q))A^\mu \pa_\mu \phi$, and $|R_{LC,1}|\lesssim \varepsilon t^{-\frac 72+2\delta}$. The decay of $A^\mu A_\mu\phi$ is even better and we omit it.

Using $-\Box=\pa_\tau^2+3\tau^{-1}\pa_\tau-\tau^{-2}\triangle_y$, one can write the Klein-Gordon equation in the hyperboloidal coordinates
\begin{equation*}
    \pa_\tau^2 (\tau^\frac 32\phi)+\tau^\frac 32\phi=\tau^\frac 32(2i A^\mu \pa_\mu\phi-A^\mu A_\mu\phi)+\tau^{-\frac 12}(\triangle_y \phi+\textstyle\frac 34\phi)
\end{equation*}
Now let $\Phi=\tau^\frac 32\phi$, and $\Phi_\pm=e^{\mp i\tau}(\pa_\tau\Phi\pm i\Phi)$. Using the decay above we have $|\Phi_\pm|\lesssim \varepsilon\tau^{-1+\delta}(1-|y|^2)^{\frac 54-\delta}+\varepsilon\tau^\delta (1-|y|^2)^{\frac 94-\delta}$. We have
\begin{equation}\label{dtauPhipm}
    \pa_\tau \Phi_\pm=e^{\mp i\tau}\left(\tau^\frac 32 (2iA^\mu\pa_\mu\phi-A^\mu A_\mu \phi)+\tau^{-\frac 12}(\triangle_y \phi+\textstyle\frac 34\phi)\right).
\end{equation}
Since $A^\mu \pa_\mu \phi=(A^\tau \pa_\tau \phi+R_{tan,1})\chi_{LC}(q)+R_{LC,1}=(\tau^{-\frac 32} A^\tau \pa_\tau \Phi-\frac 32 \tau^{-1} A^\tau \phi+R_{tan,1})\chi_{LC}(q)+R_{LC,1}$, and $\pa_\tau\Phi=\frac 12(e^{i\tau}\Phi_+ +e^{-i\tau}\Phi_-)$, we get
\begin{equation*}
    \pa_\tau\Phi_\pm =i\chi_{LC}(q)A^\tau \Phi_\pm+i e^{\mp 2i\tau}\chi_{LC}(q)A^\tau \Phi_\mp+e^{\mp i\tau}\tau^\frac 32 R,
\end{equation*}
where $R=2i\chi_{LC}(q)R_{tan,1}+\chi_{LC}(q)R_{tan,2}+2iR_{LC,1}+R_{LC,2}+R_{good}$, with $R_{tan,2}:=\tau^{-2}(\triangle_y \phi+\frac 34\phi)$, $R_{LC,2}=(1-\chi_{LC}(q))R_{tan,2}$, and $R_{good}=-A^\mu A_\mu\phi-3i\tau^{-1} A^\tau\phi$. We have the estimate $|\chi_{LC}(q)R_{tan,2}|\lesssim \varepsilon t^{-\frac 52+\delta}(t-r)\tau^{-2}\lesssim \varepsilon\tau^{-\frac 72+\delta} (1-|y|^2)^{\frac 94-\frac\delta 2}$, and $|R_{LC,2}|\lesssim \varepsilon\tau^{-2} t^{-\frac 52+\delta}\lesssim \varepsilon\tau^{-\frac 72+\delta} (1-|y|^2)^{\frac 54-\frac \delta 2}$.
This gives, e.g., for $\Phi_+$ that (for simplicity denote $\chi=\chi_{LC}(q)$)
\begin{equation*}
    \pa_\tau \left(e^{-i\int \chi A^\tau d\tau}(\Phi_+ +\frac 12 \chi A^\tau e^{-2i\tau}\Phi_-)\right)=\frac 12 e^{-2i\tau} \pa_\tau (e^{-i\int \chi A^\tau d\tau} \chi A^\tau \Phi_-)+\tau^\frac 32 e^{-i\tau-i\int \chi A^\tau d\tau} R.
\end{equation*}
We have 
\begin{equation*}
    \begin{split}
        |\pa_\tau(\chi A^\tau)\Phi_-|&=|(\pa_\tau \chi_{LC}(q))A^\tau \Phi_-|+|\chi_{LC}(q)\frac{x_\mu}\tau \pa_\tau A_\mu \Phi_-|\lesssim \varepsilon^2 t\tau^{-1}\c t^{-\frac 12} \tau^{-\frac 32+2\delta}\c \tau^\frac 32 t^{-\frac 52+\delta}(t-r)\\
        &\lesssim \varepsilon^2\tau^{-2+3\delta} (1-|y|^2)^{\frac 32-\frac\delta 2}
    \end{split}
\end{equation*}
using the bounds we established for $\pa_\tau A_\mu$.
When $\pa_\tau$ falls on $\Phi_-$, we can use \eqref{dtauPhipm} to see it behaves well. Overall, from all the estimates, we see that the right hand side of the equation can be bounded by $\varepsilon^2\tau^{-\frac 32+2\delta}(1-|y|^2)^{\frac 32-\delta}+\varepsilon\tau^{-2+2\delta}(1-|y|^2)^{\frac 54-\delta}$. 

Then, integrating this equation, we get $\Phi_+=e^{i\int \chi A^\tau d\tau} b_+(y)+O(\varepsilon\tau^{-1+3\delta}(1-|y|^2)^{\frac 54-\delta})$, where
\begin{multline}
    b_+(y)=\int_{\tau_0}^\infty \frac 12 e^{-2i\tau} \pa_\tau (e^{-i\int \chi A^\tau d\tau} \chi A^\tau \Phi_-)+\tau^\frac 32 e^{-i\tau-i\int \chi A^\tau d\tau} R d\tau\\
    +\left(e^{-i\int \chi A^\tau d\tau}(\Phi_+ +\frac 12 \chi A^\tau e^{-2i\tau}\Phi_-)\right)|_{\tau=\tau_0}.
\end{multline}
Note that along the hyperboloid $\{t^2-r^2=\tau_0^2\}$ we have $t^{-2}\sim (1-|y|^2)$. Then since we have $|Z^I\phi|\lesssim \varepsilon t^{-\frac 52+\delta}$ when $|t-r|$ is bounded, we have that along the initial hyperboloid $|\Phi_\pm|\lesssim \varepsilon(1-|y|^2)^{\frac 54-\delta}$. Therefore we can derive $|b_+(y)|\lesssim \varepsilon (1-|y|^2)^{\frac 54-\delta}$.

\begin{remark}
    In view of the expression of $b_+(y)$ and the fact that $\Om_{ij}$ and $\Om_{0i}$ are tangent to the hyperboloids (and behaves well when falling on the cutoff $\chi_{LC}(q)$), we can derive similar bounds of derivatives of $b_+(y)$:
    \begin{equation}
        (1-|y|^2)^{|I|} \nab_y^I\, b_+(y)\lesssim \varepsilon (1-|y|^2)^{\frac 54-\delta}.
    \end{equation}
\end{remark}

Similarly we have $\Phi_- =e^{-i\int A^\tau d\tau}b_-(y)+O(\tau^{-1+3\delta} (1-|y|^2)^{\frac 54-\frac \delta 2})$ with $b_-(y)$ satisfying the same bounds. Then using $\Phi=-\frac i2(e^{i\tau}\Phi_+-e^{-i\tau}\Phi_-)$ and $\pa_\tau\Phi=\frac 12(e^{i\tau}\Phi_+ +e^{-i\tau}\Phi_-)$, we have
\begin{equation}
    \begin{split}
        &\phi\sim -\frac i2\tau^{-\frac 32}(e^{i\tau+i\int \chi A^\tau d\tau} b_+(y)-e^{-i\tau-i\int A^\tau d\tau} b_-(y)),\\
        &\pa_\tau \phi\sim \frac 12\tau^{-\frac 32}(e^{i\tau+i\int \chi A^\tau d\tau} b_+(y)+e^{-i\tau-i\int A^\tau d\tau} b_-(y)).
    \end{split}
\end{equation}
In order to make the expansion of $\phi$ look simpler, we define $a_\pm(y):=\mp\frac i2 b_\pm(y)$. Then 
\begin{equation}
    \begin{split}
        &\phi\sim \tau^{-\frac 32}(e^{i\tau+i\int \chi A^\tau d\tau} a_+(y)+e^{-i\tau-i\int \chi A^\tau d\tau} a_-(y)),\\
        &\pa_\tau \phi\sim i\tau^{-\frac 32}(e^{i\tau+i\int \chi A^\tau d\tau} a_+(y)-e^{-i\tau-i\int \chi A^\tau d\tau} a_-(y)).
    \end{split}
\end{equation}

\subsubsection{The gauge potential}
We now turn to the wave equations $-\Box \tilde A_\mu=\Im(\phi\c\overline{D_\mu\phi})$ and derive the asymptotics of $A_\mu$. We have 
\begin{equation*}
    \Im (\phi\c\overline{D_\mu\phi})=-\frac{x_\mu}\tau \Im (\phi\c \overline{\pa_\tau \phi})+R_{cub}+R_{tan,w},
\end{equation*}
where $R_{cub}=O(|A|\cdot |\phi|^2)$, $R_{tan,w}=O(|\phi|\cdot t^{-1}|\nab_y \phi|)$.
Plugging the asymptotics of the scalar field above, we have
\begin{equation*}
    \begin{split}
        \Im&(\phi\c\overline{\pa_\tau\phi}) \\
        &\approx \Im\left(-i\tau^{-3}(e^{i\tau+i\int A^\tau d\tau} a_+(y)+e^{-i\tau-i\int A^\tau d\tau} a_-(y))(e^{-i\tau-i\int A^\tau d\tau} \overline{a_+(y)}-e^{i\tau+i\int A^\tau d\tau} \overline{a_-(y)})\right)\\
        &=\Im\left(-i\tau^{-3} \left(|a_+(y)|^2-|a_-(y)|^2-e^{2i(\tau+\int A^\tau d\tau)}a_+(y)\overline{a_-(y)}+e^{-2i(\tau+\int A^\tau d\tau)}a_-(y)\overline{a_+(y)}\right)\right)\\
        &=-\tau^{-3}(|a_+(y)|^2-|a_-(y)|^2),
    \end{split}
\end{equation*}
where the remainder (from the ``$\approx$") $R_{KG}=O(\varepsilon^2\tau^{-3} (1-|y|^2)^{\frac 54-\delta}\tau^{-1+3\delta} (1-|y|^2)^{\frac 54-\frac \delta 2})=O(\varepsilon^2 t^{-4+3\delta})$. Also, while the expansion is only applicable when $\tau\geq \tau_0$, we use this expression everywhere when $t\geq r$, as one can see that the error $R_{source}$ is supported in a region near the light cone $\{t\geq r,\, t-r\leq \tau_0 (t+r)^{-1}\}$, and hence decaying at the rate better than $t^{-\frac 72}$ using the boundedness of $t-r$. Therefore the equation becomes
\begin{equation*}
    -\Box \tilde A_\mu=\frac{x_\mu}\tau \tau^{-3} (|a_+(y)|^2-|a_-(y)|^2)+R_{cub}+R_{tan,w}+R_{KG}+R_{source}.
\end{equation*}
Again, by linearity, we consider vanishing initial data first. Notice that $\frac{x_\mu}\tau$ is only dependent on $y$ variables. Therefore the analysis of the main terms is in fact on the equation
\begin{equation*}
    -\Box A^M_\mu=t^{-3} F_\mu(y)
\end{equation*}
with vanishing data (at $t=2$), where $F_\mu(y)=(\frac t\tau)^3 \frac{x_\mu}\tau (|a_+(y)|^2-|a_-(y)|^2)$. These functions in $y$ are defined for $|y|\leq 1$, and we understand them as zero when $|y|>1$. We have studied this equation in detail in \cite{CL}: Using the representation formula, the solution can be written as
\begin{equation}\label{lambdaformula}
\begin{split}
    A^M_\mu(t,x)&=\frac 1{4\pi}\int_2^t \frac{1}{t-s} \int_{|\bar{z}|=t-s} s^{-3} F_\mu\left (\frac{x-\bar z}{s}\right ) \, d\sigma(\bar z) ds\\
    &=\frac 1{4\pi} \int_{\frac 2t}^1 \frac{1}{1-\lambda} \int_{|\bar z|=(1-\lambda)t} (\lambda t)^{-3} F_\mu\left (\frac{x-\bar z}{\lambda t}\right )\, d\sigma(\bar z) d\lambda\\
    &=\frac 1{4\pi}\int_{\frac 2t}^1 \int_{\mathbb{S}^2} (1-\lambda)t^{-1}\lambda^{-3} F_\mu\left (\frac{x/t-(1-\lambda)\eta}{\lambda}\right )\, d\sigma(\eta) d\lambda
    \end{split}
\end{equation}
Notice that if $\lambda$ near $0$ satisfies $1-\lambda-r/t>\lambda$, i.e., $\lambda<\frac 12(1-\frac rt)$, then the integrand is zero because of the support of $F_\mu$. Therefore, if $2/t<\frac 12(1-r/t)$, i.e., $t-r>4$, we can replace the lower bound of the integral $2/t$ by $\frac 12(1-r/t)$. As a result, in the region $\{t-r>4\}$ we have $A^M_\mu=t^{-1} \widetilde U_\mu(x/t)$, where $$\widetilde U_\mu(y):=\frac 1{4\pi}\int_{\frac 12(1-|y|)}^1 \int_{\eta\in \mathbb{S}^2} (1-\lambda) \lambda^{-3} F_\mu\left (\frac{y-(1-\lambda)\eta}{\lambda}\right )\, d\sigma(\eta) d\lambda.$$
Therefore, for $U_\mu(y)=\frac \tau t \widetilde U_\mu(y)$, we have
\begin{equation*}
    A^M_\mu=\frac{U_\mu(y)}\tau, \quad t-r\geq 4.
\end{equation*}
For the remaining inhomogeneous terms, we consider
\begin{equation*}
    -\Box A^R_\mu=R_{cub}+R_{tan,w}+R_{good}+R_{LC}
\end{equation*}
with vanishing data. Using the decay estimates we can bound all of them by $\varepsilon^2(1+t+r)^{-4+3\delta}$. This implies
\begin{equation*}
    |A^R_\mu|\lesssim \varepsilon^2 (1+t+r)^{-1}(1+|t-r|)^{-1+3\delta}, \text{ so } |A^R_\mu|\lesssim \varepsilon^2\tau^{-2+3\delta}(1-|y|^2)^{\frac {3}2 \delta} \text{ when $t-r\geq 1$.}
\end{equation*}
Now for the initial data part, we have already derived the estimate of its contribution $\varepsilon(1+t+r)^{-1}(1+|t-r|)^{-\frac{\gamma_0'-1}2}\lesssim \varepsilon\tau^{-\frac{\gamma_0'+1}2}(1-|y|^2)^{\frac{3-\gamma_0'}4}$. So combining these together we have (note that $\gamma_0'>1$)
\begin{equation*}
    A_\mu=\frac{U_\mu(y)}\tau+O(\varepsilon\tau^{-\frac{\gamma_0'+1}2}(1-|y|^2)^{\frac{3-\gamma_0'}4}), \quad t-r\geq 4.
\end{equation*}
Therefore, we also have
\begin{equation*}
    A^\tau=\frac{U(y)}\tau+O(\varepsilon\tau^{-\frac{\gamma_0'+1}2}(1-|y|^2)^{-\frac{\gamma_0'-1}2}),\quad t-r\geq 4,
\end{equation*}
so for the phase in the expansion of the complex scalar field, we have $\int \chi_{LC}(q) A^\tau d\tau= U(y)\ln\tau+h$, where $h=O(\tau^{-1}(1+(t-r))^{-\frac{\gamma_0'}2})$ is uniformly bounded in the region $\{t-r\geq 4\}$.

\vspace{0.5ex}

Note that if we denote the initial data part (linear part) by $A_\mu^{Linear}$ and write them as an independent part in the expansion, then the remainder has better estimates as in Theorem \ref{mainthm1}.

\subsubsection{Radiation fields}
Now we discuss the behavior of $A_\mu$ towards null infinity. While the lower order terms above decay well towards timelike infinity, one cannot expect the decay to be better than $r^{-1}$ along the light cone. In fact, since each of them satisfies a wave equation, the radiation fields all exist. For example we have
\begin{prop}
    The limit $\lim_{r\rightarrow \infty} rA^M_\mu(r,q,\om)$ exists for each $q(=r-t)$ and $\om$, which we denote by $F^M_\mu(q,\om)$. We also have the bound
    \begin{equation}
        |rA^M_\mu-F^M_\mu|\lesssim \varepsilon r^{-1} q_-
    \end{equation}
    where $q_-=\max\{-q,0\}$.
\end{prop}
\begin{proof}
    Recall the equation $-\Box A^M_\mu=\frac{x_\mu}\tau \tau^{-3}(|a_+(y)|^2-|a_-(y)|^2)$. Since the bound of $a_+(y)$ is similar with vector fields, we can commute the equation with $\triangle_\om=\sum\Om_{ij}^2$ to derive the decay of $\triangle_\om A^M_\mu$. Then one can use the decomposition $-\Box A^M_\mu=r^{-1}(\pa_t-\pa_r)(\pa_t+\pa_r)(r A^M_\mu)+r^{-2}\triangle_\om A^M_\mu$. Integrating in $-(\pa_t-\pa_r)$ direction to the initial slice gives (when $t<r$ the source is zero)
    \begin{equation*}
        |(\pa_t+\pa_r)(rA^M_\mu)|\lesssim \varepsilon r^{-1} (1+t+r)^{-1} q_-.
    \end{equation*}
    Now integrating in $\pa_t+\pa_r$ direction for $r\in [r_1,r_2]$, we get
    \begin{equation*}
        |r_1 A^M_\mu(r_1,q,\om)-r_2 A^M_\mu(r_2,q,\om)|\lesssim \varepsilon r_1^{-1} q_-,
    \end{equation*}
    which shows the existence of the limit. Letting $r_2\rightarrow\infty$ also gives
    \begin{equation*}
        |r A^M_\mu-F^M_\mu(q,\om)|\lesssim \varepsilon r^{-1} q_-
    \end{equation*}
    as required.
\end{proof}
Since $A^M_\mu=U_\mu(y)/\tau$ when $t-r\geq 4$, the proposition above also shows the existence of the limit $\lim_{|y|\rightarrow 1^-}(1-|y|^2)^{-\frac 12} U_\mu(y)$. For each $\mu=0,1,2,3$, the limit is a function of $\om$ and we denote it as $\mathcal U_\mu(\om)$. Then we have $F^M_\mu(q,\om)=\mathcal U_\mu(\om)$ when $q<-4$.

One can similarly prove the existence of the radiation field for other parts ($A^R_\mu$ and the initial data part). However, due to additional decay in $\langle t-r\rangle$, these radiation fields is decaying in $\langle q \rangle$ as $q\rightarrow \pm\infty$. Then, also taking into account of the charge part $\q_0 \chi_{ex}(r-t) r^{-1}\delta_{0\mu}$, we have the radiation field of $A_\mu$, denoted by $F^{total}_\mu(q,\om)$, satisfies
\begin{equation}
    F^{total}_\mu(q,\om)=\mathcal U_\mu(\om)\mathbf 1_{q<0}+\q_0 \delta_{0\mu}\mathbf 1_{q>0}+O(\varepsilon\langle q\rangle^{-\frac{\gamma_0'-1}2}).
\end{equation}

\begin{remark}
    We can also compute the convergence rate to the radiation fields for other parts. For example for the linear (initial data) part, we can proceed similarly as above to get
    \begin{equation*}
        |(\pa_t-\pa_r)(\pa_t+\pa_r)(A^{Linear}_\mu)|\lesssim \varepsilon r^{-1}(1+t+r)^{-1}(1+|q|)^{-\frac{\gamma_0'-1}2}
    \end{equation*}
    Integrating this to the initial slice means integrating $q$ between $r-t$ and $r+t$, which gives
    \begin{equation*}
        |(\pa_t+\pa_r)(rA^{Linear}_\mu)|\lesssim \varepsilon r^{-1}(1+t+r)^{-1}((1+|q|)^{1-\frac{\gamma_0'-1}2}+(1+t+r)^{1-\frac{\gamma_0'-1}2})
    \end{equation*}
    which gives 
    \begin{equation}
        |A^{Linear}_\mu-\chi(\frac\qq r)F^{Linear}_\mu(q,\om)|\lesssim \varepsilon (1+t+r)^{-1-\frac{\gamma_0'-1}2}.
    \end{equation}
\end{remark}
\vspace{1ex}

Nevertheless, this is not the whole story. The gauge condition itself implies additional structure: if we decompose the Lorenz gauge condition $\pa^\mu A_\mu$ in the null frame $\{L,\Lb,e_1,e_2\}$, we see that $\Lb A_L\sim -L A_{\Lb}+2(e_B A_\mu)(e_B)^\mu$. Therefore, the $\Lb$ derivative of some components of $A$, generally expected to be the ``bad" derivative, behaves like tangential derivatives (to the outgoing light cone). Therefore, the $A_L$ component may decay at a rate better than $r^{-1}$. This is observed in \cite{CKL} for the massless MKG system in the Lorenz gauge. We now prove that this also holds for the massive system.

\begin{prop}
    The radiation fields $F_\mu(q,\om)$ satisfies $F_L(q,\om)=\q_0$ for all $q$ and $\om$.
\end{prop}
\begin{proof}
We have
\begin{equation}
    \begin{split}
        L(r\Lb (rA_L)+r\tilde A_\Lb)&=rL\Lb(rA_L)+\Lb(r A_L)+L(rA_{\Lb}-\q_0 \chi_{ex}(r-t))\\
        &=-r^2 (\Box A_\mu)L^\mu+(\triangle_\om A_\mu)L^\mu+r\Lb A_L+rLA_{\Lb}-A_\mu (\pa_r)^\mu\\
        &=r^2 J_L+2r(e_B A_\mu)(e_B)^\mu+(\triangle_\om A_\mu)L^\mu-A_\mu(\pa_r)^\mu.
    \end{split}
\end{equation}
Notice that $A_\mu (\pa_r)^\mu=\tilde A_\mu (\pa_r)^\mu$, $e_B A_\mu=e_B\tilde A_\mu$, $\triangle_\om A_\mu=\triangle_\om \tilde A_\mu$, and $|re_B \tilde A_\mu|\lesssim \sum_{i<j}|\Om_{ij} A_\mu|$. Then using the decay estimates we have
\begin{equation*}
    |L(r\Lb (rA_L)+r\tilde A_\Lb)|\lesssim \varepsilon(1+t+r)^{-1+\delta}(1+q_+)^{-\frac{\gamma_0'-1}2}.
\end{equation*}
Integrating in $\pa_t+\pa_r$ direction backward to $r=0$ or $t=0$ gives
\begin{equation*}
    |r\Lb(r A_L)+r\tilde A_\Lb|\lesssim \varepsilon(1+t+r)^{\delta} (1+q_+)^{-\frac{\gamma_0'-1}2}.
\end{equation*}
Since $|\tilde A_L|\lesssim \varepsilon(1+t+r)^{-1+\delta}(1+q_+)^{-\frac{\gamma_0'-1}2}$, we get
\begin{equation*}
    |\Lb(rA_L)|\lesssim \varepsilon r^{-1}(1+t+r)^{\delta}(1+q_+)^{-\frac{\gamma_0'-1}2}.
\end{equation*}
Now integrate in $\pa_r-\pa_t$ direction to the initial slice. Recall that at the initial slice
\begin{equation*}
    |rA_L-\q_0|\lesssim \varepsilon (1+r)^{-\frac{\gamma_0'-1}2},
\end{equation*}
so we have
\begin{equation*}
    \begin{split}
        |rA_L-\q_0|&\leq |rA_L(t,r)-rA_L(0,t+r)|+|rA_L(0,t+r)-\q_0|\\
        &\lesssim \int_{r-t}^{r+t} \varepsilon r^{-1}(1+t+r)^{\delta}(1+q_+)^{-\frac{\gamma_0'-1}2} dq+\varepsilon (1+t+r)^{-\frac{\gamma_0'-1}2}\\
        &\leq \varepsilon r^{-1}(1+t+r)^\delta \int_{r-t}^{r+t} (1+q_+)^{-\frac{\gamma_0'-1}2} dq+\varepsilon (1+t+r)^{-\frac{\gamma_0'-1}2}\\
        &\lesssim \varepsilon r^{-1}(1+t+r)^{1+\delta-\frac{\gamma_0'-1}2}+\varepsilon (1+t+r)^{-\frac{\gamma_0'-1}2}.
    \end{split}
\end{equation*}
Note that $\delta$ is a small number, so we have $\delta<\frac{\gamma_0'-1}2$. Therefore, we have $rA_L\rightarrow \q_0$ as $r\rightarrow\infty$ for every $q=r-t$ fixed. This shows that the radiation fields $F^{total}_\mu$ satisfies $F^{total}_L=\q_0$.
\end{proof}

\subsection{Charge at infinity}
Recall the charge conservation identity $\pa^\mu J_\mu=0$. We can use this to derive an expression of charge using the information at infinity. 

Consider the region enclosed by the initial slice, a part of an incoming null cone $\underline C_{\tau^2}=\{t+r=\tau^2,\, t-r\leq 1\}$, and the hyperboloid $\hat H_\tau$ (truncated at $t-r=1$) to get
\begin{equation*}
    \int_{\hat H_\tau} \Im (\phi\c\overline{D_\tau \phi})\, dH_\tau+\int_{\underline C_{\tau^2}} \Im(\phi\c\overline{D_{\Lb} \phi})\, dvol=4\pi \q_0.
\end{equation*}
Using the decay estimates, it is easy to see that the second term on the left is decaying to zero as $\tau \rightarrow\infty$. On the other hand, we have shown above that in the interior
\begin{equation*}
   \Im(\phi\c\overline{D_\tau\phi})=-\tau^{-3}(|a_+(y)|^2-|a_-(y)|^2)+O(t^{-\frac 72+\delta}),
\end{equation*}
so using $dH_\tau=\tau^3 dH_1$, we have for the first integral
\begin{equation*}
\begin{split}
    \int_{\hat H_\tau} \Im(\phi\c\overline{D_\tau\phi})\, dH_\tau&=\int_{\hat H_\tau} (\tau^{-3}(|a_+(y)|^2-|a_-(y)|^2)+O(t^{-\frac 72+\delta})) \tau^3 dH_1\\
    &\rightarrow\int_{H_1} |a_-(y)|^2-|a_+(y)|^2 \, dH_1 \quad \text{ as }\tau\rightarrow\infty. 
\end{split}
\end{equation*}
Therefore
\begin{equation}\label{q0atinfinity}
    \q_0=\frac{1}{4\pi}\int_{H_1} |a_-(y)|^2-|a_+(y)|^2 \, dH_1.
\end{equation}


\section{Scattering from infinity}
The goal of this section is to study the scattering from infinity problem of the mMKG system. We will first consider the system in the Lorenz gauge:
\begin{equation}
    -\Box A_\mu=\Im(\phi\c\overline{D_\mu\phi}),\quad -\Box\phi+\phi=2iA^\mu\pa_\mu\phi-A^\mu A_\mu\phi,
\end{equation}
and then show that for admissible scattering data, the solution we get in fact also solves the original mMKG system.

In the forward problem, we showed that
\begin{equation*}
    A_\mu \sim U_\mu(y)/\tau, \quad \phi\sim \tau^{-\frac 32}(e^{i\tau+iU(y)\ln \tau+ih}a_+(y)+e^{-i\tau-iU(y)\ln\tau-ih}a_-(y))
\end{equation*}
in the interior, where $U(y):=U^\tau (y)=-\frac{x_\mu}\tau U_\mu(y)$. The functions $U_\mu(y)$ are in fact determined by $a_\pm(y)$, and $R$ is uniformly bounded in $\{t-r\geq 1\}$. We also have the radiation fields $F_\mu(q,\om)=\lim_{r\rightarrow\infty} rA_\mu(r,q,\om)$. In particular in the forward problem we have $rA_L\rightarrow \q_0$ where $\q_0$ is the charge defined by the initial data.

Now for the backward problem, we consider similar type of asymptotic behaviors to be given. Since we want the Lorenz gauge condition to hold, we expect a similar relation on the radiation fields of the solution we construct. In view of the charge relation \eqref{q0atinfinity}, we define the charge at infinity using the scattering data
\begin{equation}
    \q_\infty :=\frac{1}{4\pi}\int_{H_1} |a_-(y)|^2-|a_+(y)|^2 dH_1.
\end{equation}
As in the forward problem, $a_+(y)$ and $a_-(y)$ determine the main part of the wave components in the interior. Hence we define $A^M_\mu$ to solve
\begin{equation}
    -\Box A^M_\mu=\frac{x_\mu}\tau \tau^{-3}(|a_+(y)|^2-|a_-(y)|^2)
\end{equation}
with vanishing data at $\{t=2\}$. Then we know that when $t-r>4$ we have $A^M_\mu=U_\mu(y)/\tau$, so $U_\mu(y)$ are determined by $a_+(y)$ and $a_-(y)$.

\begin{remark}
One may also add a $\chi_{LC}(q)$ cutoff factor on the right hand side of $-\Box A^M_\mu$ (we used $\chi_{LC}(q)$ in the forward problem on the phase; it is $1$ when $q\leq -1$ and $0$ when $q\geq -1/2$). It is not hard to see that this differs from the one without cutoff by a part supported near the light cone, which decays very well and affects little in the analysis. Also, we can see that when we commute the equation with vector fields, terms with vector fields falling on $\chi_{LC}(q)$ gives similarly good terms supported in $-1\leq q\leq 0$ as $|Zq|\leq |t-r|\leq 1$. For similar reasons, one can also add this cutoff factor on the phase correction in front of $(\az)^\tau$ below. The advantage of doing this is that the approximate solution behaves more regularly near the light cone. However, to make things look simpler we here proceed without the cutoff, but these modifications are straightforward.
\end{remark}

We can now define the approximate solutions
\begin{equation}
    \az_\mu=A^M_\mu+\chi(\frac\qq r) \left(\frac{F_\mu(q,\om)}r+\frac{F^{(1)}_\mu(q,\om)}{r^2}\right)+\q_\infty \chi_{ex}(q) \delta_{0\mu} r^{-1},
\end{equation}
\begin{equation}
    \phiz=\tau^{-\frac 32}(e^{i\tau+i\int (\az)^\tau d\tau} a_+(y)+e^{-i\tau-i\int (\az)^\tau d\tau} a_-(y)).
\end{equation}
Here $\chi_{ex}(q)$ supported in $\{q\geq 1\}$ is $1$ when $q\geq 2$. The fields $F_\mu(q,\om)$ should be understand as the total radiation fields subtracted by the contribution of $A^M_\mu$ and $\chi_{ex}(q)\, \q_\infty \delta_{0\mu}r^{-1}$, and we require them to decay in $\langle q\rangle$:
\begin{equation}
    \sum_{k+|\b|\leq N}|(\langle q\rangle \pa_q)^k \pa_\om^\b F_\mu(q,\om)|\lesssim \varepsilon\langle q\rangle^{-1+\gamma},\quad \gamma<1/2,\, N\geq 6.
\end{equation}
Later we will see that $F_\mu(q,\om)$ needs to satisfy a compatibility condition \eqref{admissiblecondition} which essentially comes from the Lorenz gauge condition.
The second order approximation $F^{(1)}(q,\om)$ is defined similarly as in \cite{LindbladSchlue1} and \cite{CL}:
\begin{equation}
    2\pa_q F^{(1)}_\mu(q,\om)=\triangle_\om F_\mu(q,\om).
\end{equation}
We also impose the decay assumptions on $a_\pm(y)$:
\begin{equation}
    \sum_{|I|\leq N}|(1-|y|^2)^{|I|}\nabla_y^I a_\pm(y)|\lesssim \varepsilon(1-|y|^2)^\a, \quad \a\geq 7/4,
\end{equation}
which implies that $a_\pm(y)$ decay at the same rate when applied with rotation and boost vector fields. 
\begin{remark}
    When we apply partial derivatives to $a_\pm(y)$, using this condition, we get an additional $\tau^{-1}(1-|y|^2)^{-\frac 12}$. This quantity is bounded when $t-r\geq 1$, and we see that one advantage of the cutoff in the previous remark is that we can allow more partial derivatives in the estimate.
\end{remark}

In the forward problem, we see that the Lorenz gauge condition is closely related to the fact that $L$ component of the radiation fields to be constant (the charge). From the expression of $\az$ we see that the value is $\q_\infty$ when $q\rightarrow +\infty$. Therefore, to make the approximate solutions a good candidate not only for the reduced system but also for the original system, a crucial step is to show that when we consider the radiation fields of $A^M_\mu$, denoted by $F^M_\mu(q,\om)$, we have $F^M_L(q,\om)\rightarrow \q_\infty$ as $q\rightarrow -\infty$. We shall address this question in the following subsection, and show that $rA^M_L\rightarrow \q_\infty$ as $r\rightarrow \infty$ with $q=r-t<-4$ fixed. 

In later subsections, We shall first construct solutions to the reduced system, and then show that the solutions also solve the original mMKG system. The latter step relies heavily on the decay property of $\pa^\mu\az_\mu$, which requires, again, the $L$ component of the radiation field is always $\q_\infty$. Therefore, the coefficient of the exterior part must be $\q_\infty$, because otherwise one would not be able to get a solution to the original system.


\subsection{Computing the radiation field of the interior solution}\label{sectionradiationfieldcharge}
Recall that 
\begin{equation*}
    -\Box A^M_\mu= \frac{x_\mu}\tau \tau^{-3} (|a_+(y)|^2-|a_-(y)|^2)=:x_\mu t^{-4} P(y)
\end{equation*}
and the data is vanishing at $t=2$.\footnote{One can pick any positive number $\varepsilon>0$ instead of $2$ here, and then the part where $q<-\varepsilon$ will be exactly $U_\mu(y)/\tau$. This is simply a consequence of strong Huygens' principle.} We also denote $F(y):=|a_+(y)|^2-|a_-(y)|^2=\tau^4 t^{-4} P(y)=(1-|y|^2)^2 P(y)$.

When $q<-4$, we have $A^M_\mu=U_\mu(y)/\tau$. The form implies that the radiation fields $F^M_\mu(q,\om)$ will in fact be independent of $q$ when $q<-4$. Therefore, we only need to fix one $q<-4$ in the computation.

When $t-r>4$, the representation formula gives
\begin{equation*}
    A^M_\mu=\frac{1}{4\pi}\int_2^t \frac 1{t-s}\int_{\mathbb{S}^2} \left(s,x-(t-s)\eta\right)_\mu\,  s^{-4} P\left(\frac{x-(t-s)\eta}{s}\right)(t-s)^2 d\sigma(\eta) ds
\end{equation*}
Here $(s,x-(t-s)\eta)_0=-s$, and $(s,x-(t-s)\eta)_i=(x-(t-s)\eta)_i$. Since $x=r\om$, contracting with $L^\mu$ we get
\begin{equation*}
    A^M_L=\frac{1}{4\pi}\int_2^t \int_{\mathbb{S}^2} (-s+r-(t-s)\langle\eta,\om\rangle)\,  s^{-4} P\left(\frac{x-(t-s)\eta}{s}\right)(t-s) d\sigma(\eta) ds.
\end{equation*}
Recall that this is an integral on the part of a backward light cone with the tip at $(t,x)$ restricted in the part $t\geq 2$. In the case when $t-|x|>4$, for all angles, the backward cone already hits the support boundary $\{t=|x|\}$ at some time greater than $2$. We can compute the value for each $\eta$: Consider the distance $R=R(\eta)$ so that $(t,r,0,0)-(R,R\eta_1,R\eta_2,R\eta_3)$ is on the cone $\{t=r\}$. Then
\begin{equation*}
    (t-R)^2=|x-R\eta|^2\quad \implies \quad R(\eta)=\frac{t^2-r^2}{2(t-r\langle\eta,\om\rangle)}.
\end{equation*}
We now consider the change of variable $R=t-s$. This gives 
\begin{equation*}
    A^M_L=\frac{1}{4\pi}\int_{\mathbb{S}^2} \int_0^{R(\eta)} (R-t+r-R\langle \eta,\om\rangle)\,  (t-R)^{-4} P\left(\frac{x-R\eta}{R}\right) R\, dR d\sigma(\eta).
\end{equation*}
We want to study the limit of $tA_L^M$ as $t\rightarrow\infty$ with $q=r-t$ fixed. When $q<-4$, we have
\begin{equation*}
    tA^M_L=\frac{1}{4\pi}\int_{\mathbb{S}^2} \int_0^{R(\eta)} tR(R+q-R\langle\eta,\om\rangle)\,  (t-R)^{-4} P\left(\frac{x-R\eta}{R}\right) dR d\sigma(\eta).
\end{equation*}



\begin{lem}
    Given $t$ and $x$ with $|x|<t$. Then for a function defined on the unit hyperboloid (equipped with the induced metric from Minkowski) $P(y)$ and $F(y)=(1-|y|^2)^2 P(y)$, we have the following relation
    \begin{equation}
        \int_{H_1} F(y) \, dH_1=\int_{\mathbb{S}^2} \int_0^{R(\eta)} P\left(\frac{x-R\eta}{t-R}\right) \frac{R^2(t-r\langle\eta,\om\rangle)}{(R-t)^4} dR d\sigma(\eta)
    \end{equation}
    where $R(\eta)=(t^2-r^2)/(2(t-r\langle\eta,\om\rangle))$ is defined as above.
\end{lem}
\begin{proof}
    Without loss of generality, we can assume $\om=(1,0,0)$ in proving this lemma, so $\langle \eta,\om\rangle=\eta_1$. We first write the integral on the unit hyperboloid $H_1$. In the coordinate $(|y|,\theta,\varphi)$, the integral can be written as
\begin{equation*}
    \int_{\mathbb{S}^2} \int_0^1 F(y) \frac 1{(1-|y|^2)^{2}} |y|^2 d|y| d\mathbb{S}^2=\int_{\mathbb{S}^2} \int_0^1 P(y) |y|^2 d|y| d\mathbb{S}^2.
\end{equation*}
Now if we denote the standard induced Euclidean metric on unit ball (including interior) by $dy$, then the integral equals
\begin{equation*}
    \int_{|y|\leq 1} P(y)\, dy.
\end{equation*}
We wish to switch to a coordinate system more adapted to the integral we need to deal with. For every given $x=(r,0,0)$ and $t$, all points in the unit ball can be uniquely expressed in $R$ and $\eta$ as
\begin{equation*}
    y=\frac{x-R\eta}{t-R}
\end{equation*}
where $0\leq R \leq R(\eta)$. Now we want to write the integral in $R$ and $\eta$. We use the coordinate $(R,\eta_2,\eta_3)$, so we need to split the case $\eta_1\geq 0$ and $\eta_1<0$. When $\eta_1\geq 0$, we have $\eta_1=\sqrt{1-\eta_2^2-\eta_3^2}$. Then the coordinate change reads
\begin{equation*}
    y_1=\frac{r-R\sqrt{1-\eta_2^2-\eta_3^2}}{t-R},\quad y_2=-\frac{R\eta_2}{t-R},\quad y_3=-\frac{R\eta_3}{t-R}.
\end{equation*}
One can then compute the Jacobian
\begin{equation*}
    \det\frac{\pa(y_1,y_2,y_3)}{\pa(R,\eta_2,\eta_3)}=\frac{R^2(r\sqrt{1-\eta_2^2-\eta_3^2}-t)}{(R-t)^4 \sqrt{1-\eta_2^2-\eta_3^2}}.
\end{equation*}
So the part $\eta_1\geq 0$ can be written as
\begin{equation*}
    \int_0^1 \int_0^1 \int_0^{R(\eta)} P\left(\frac{x-R\eta}{t-R}\right) \left|\frac{R^2(r\sqrt{1-\eta_2^2-\eta_3^2}-t)}{(R-t)^4 \sqrt{1-\eta_2^2-\eta_3^2}}\right|\, dR d\eta_2 d\eta_3
\end{equation*}
The case when $\eta_1<0$ in fact gives the same expression. Also $\frac 1{\sqrt{1-\eta_2^2-\eta_3^2}}d\eta_2 d\eta_3=d\mathbb{S}^2=d\sigma(\eta)$, so the integral becomes
\begin{equation*}
    \int_{\mathbb{S}^2} \int_0^{R(\eta)} P\left(\frac{x-R\eta}{t-R}\right) \frac{R^2(t-r\eta_1)}{(R-t)^4} dR d\sigma(\eta)
\end{equation*}
which concludes the proof.
\end{proof}

Recall \begin{equation}
    \q_\infty=\frac{1}{4\pi}\int \left(|a_-(y)|^2-|a_+(y)|^2\right)\, dH_1=-\frac{1}{4\pi}\int_{H_1} F(y)\, dH_1.
\end{equation}
Then using the lemma, we have
\begin{equation*}
    tA^M_L-\q_\infty=\frac{1}{4\pi}
    \int_{\mathbb{S}^2} \int_0^{R(\eta)} P\left(\frac{x-R\eta}{t-R}\right) \frac{1}{(R-t)^4}(tR(R+q-R\langle\eta,\om\rangle)+R^2(t-r\langle\eta,\om\rangle)) dR d\sigma(\eta)
\end{equation*}
\begin{equation*}
    =\frac{1}{4\pi}\int_{\mathbb{S}^2} \int_0^{R(\eta)} P\left(\frac{x-R\eta}{t-R}\right) \frac{1}{(R-t)^4}\left(2tR^2(1-\langle\eta,\om\rangle)+qR(t-R\langle\eta,\om\rangle)\right)\, dR d\sigma(\eta).
\end{equation*}

Now we need to consider the integral
\begin{equation*}
    \int_{\mathbb{S}^2} \int_0^{R(\eta)} P\left(\frac{x-R\eta}{t-R}\right) \frac{1}{(R-t)^4}\left(2tR^2(1-\langle\eta,\om\rangle)+qR(t-R\langle\eta,\om\rangle)\right)\, dR d\sigma(\eta).
\end{equation*}
First we point out that the part where $R\leq 1$ will have a lot of decay in $t$ due to the $(R-t)^{-4}$ factor, hence converging to zero as $t\rightarrow\infty$. Therefore, in the analysis below we only consider the part $R\geq 1$. Using $|P(y)|\leq \varepsilon$, we seek to estimate the absolute integral
\begin{equation*}
    \mathcal I=\int_{\mathbb{S}^2} \int_0^{R(\eta)} \frac{1}{(R-t)^4}\left(2tR^2(1-\langle\eta,\om\rangle)+|q|R(t-R\langle\eta,\om\rangle)\right)\, dR d\sigma(\eta).
\end{equation*}
Since the integrand is now nonnegative, we can change the order of integration. Moreover, the integral is now independent of $\om$, so we may just take $\om=(1,0,0)$. For fixed $R$, which varies from $0$ to $(t+r)/2$ (note from above, however, we only care about values greater than $1$), in view of the expression of $R(\eta)$, the range of $\langle \eta,\om\rangle=\eta_1$ becomes
\begin{equation*}
    \frac{2tR-(t^2-r^2)}{2rR}\leq \eta_1\leq 1.
\end{equation*}
Also, for the angular variable, the integrand now is only depedent on $\eta_1$, so we can rewrite it as an integral in $\eta_1$. In this case we have $d\sigma(\eta)=2\pi |d\eta_1|$. So the integral (the $R\geq 1$ part) becomes
\begin{equation*}
    2\pi\int_1^{\frac{t+r}2} \int_{\frac{2tR-(t^2-r^2)}{2rR}}^1 \frac{1}{(R-t)^4}\left(2tR^2(1-\eta_1)+|q|R(t-R\eta_1)\right)\, d\eta_1 dR.
\end{equation*}
We have $$1-\eta_1\leq \frac{(2R-(t+r))q}{2rR},\quad t-R\eta_1\leq t(1-\eta_1)+(t-R)\eta_1,$$
so can proceed the estimate as follows
\begin{equation*}
    |\mathcal I|\lesssim \int_1^{\frac{t+r}2} \int_{\frac{2tR-(t^2-r^2)}{2rR}}^1 \frac{1}{(R-t)^4}\left(2tR^2 (1-\eta_1)+|q|R\left(t(1-\eta_1)+(t-R)\eta_1\right)\right)\, d\eta_1 dR
\end{equation*}
\begin{equation*}
    \lesssim \int_1^{\frac{t+r}2} \frac{1}{(R-t)^4} \left((2tR^2+|q|Rt) \frac{(2R-(t+r))q}{2rR}+|q|R(t-R)\right) \frac{(2R-(t+r))q}{2rR}\, dR
\end{equation*}
\begin{equation*}
    \lesssim \int_1^{\frac{t+r}2} (\frac{2t|q|^2}{r^2}+\frac{t|q|^3}{r^2 R}) {(R-t)^{-4}}(R-\frac{t+r}2)^2\, dR+\int_1^{\frac{t+r}2} \frac{|q|^2}{r}{(t-R)^{-3}}(\frac{t+r}2-R) dR
\end{equation*}
($R=\frac{t+r}2-s$; $R\geq 1$)
\begin{equation*}
    \lesssim \int_0^{\frac{t+r}2-1} \frac{t(2|q|^2+|q|^3)}{(t+q)^2}(s+\frac{|q|}2)^{-4}s^2\, ds+\int_0^{\frac{t+r}2-1} \frac{|q|^2}{r}{(s+\frac{|q|}2)^{-3}}s\, ds
\end{equation*}
\begin{equation*}
    \lesssim \frac{t(2|q|^2+|q|^3)}{(t+q)^2} \int_0^\infty {(s+\frac{|q|}2)^{-4}}s^2\, ds+\frac{|q|^2}{(t+q)}\int_0^\infty {(s+\frac{|q|}2)^{-3}}s\, ds.
\end{equation*}
Note that $q$ is a fixed constant here. Now since both integrals are convergent, $|\mathcal I|$ goes to zero as $t\rightarrow \infty$. Therefore we have proved that
\begin{equation}
    tA^M_L \rightarrow \q_\infty,\quad t\rightarrow\infty \quad \text{with $q=r-t<-4$ fixed,}
\end{equation}
which, equivalently, means that
\begin{equation}
    (1-|y|^2)^\frac 12 U_L(y)\rightarrow \q_\infty.
\end{equation}
This is actually the conclusion of Proposition \ref{propintro}.

Later we will need an estimate on how fast $tA^M_\mu$ (or $rA^M_\mu$) converges to the radiation field $F^M_\mu$. We have discussed the same problem in the forward part, and one can show that
\begin{equation}\label{convergencerate}
    |rA^M_\mu-F^M_\mu|\leq C_q \varepsilon \, r^{-1}.
\end{equation}
Here $C_q$ is a constant that may grow in $\langle q\rangle$, but we will only use this estimate when $q$ is in a compact interval. Moreover, using the argument in \cite[Section 6.2]{H97}, we have that the radiation field of $Z^I A^M_\mu$ equals $rZ^I (r^{-1} F^M_\mu(q,\om))$. Then by commuting the equation of $A^M_\mu$ with vector fields, one can also prove the convergence rate with vector fields:
\begin{equation}\label{convergencerateZ}
    |rZ^I A^M_\mu-rZ^I(r^{-1} F_\mu(q,\om))|\leq C_q \varepsilon\, r^{-1}.
\end{equation}

\subsection{Constructing the approximate solutions}
\subsubsection{Condition on the scattering data}
Given scattering data $(a_+(y),a_-(y),F_\mu(q,\om))$, we can define the approximate solution
\begin{equation}
    \az_\mu=A^M_\mu+\chi(\langle q\rangle/r) (F_\mu(q,\om) r^{-1}+F^{(1)}_\mu(q,\om)r^{-2})+\chi_{ex}(q)\q_\infty r^{-1}\delta_{0\mu}.
\end{equation}
where $F_\mu(q,\om)$ satisfies
\begin{equation}
    \sum_{k+|\b|\leq N} |(\langle q\rangle \pa_q)^k \pa_\om^\beta F_\mu(q,\om)|\lesssim \varepsilon \langle q\rangle^{-1+\gamma}.
\end{equation}
Similar to \cite{LiliHe21}, we need the scattering data to satisfy an \textit{asymptotic Lorenz gauge condition}, which says that $\pa^\mu \az_\mu$ decays well, e.g., at the rate of $O(\langle t+r\rangle^{-2})$.

Now we compute this quantity. Recall the wave equations of $A^M_\mu$. Commuting them with $\pa^\mu$, we get
\begin{equation}
    \Box \pa^\mu A^M_\mu=(\pa^\mu x_\mu)\tau^{-4} (|a_-(y)|^2-|a_+(y)|^2)+x_\mu \pa^\mu (\tau^{-4} (|a_-(y)|^2-|a_+(y)|^2))
\end{equation}
Notice that $x^\mu \pa_\mu=S=\tau\pa_\tau$. Then the right hand side becomes
\begin{equation*}
    4\tau^{-4} (|a_-(y)|^2-|a_+(y)|^2)-\tau\c 4\tau^{-5} (|a_-(y)|^2-|a_+(y)|^2)=0.
\end{equation*}
We know that $A^M_\mu$ and the time derivatives are zero at $\{t=2\}$. From the equation
\begin{equation*}
    \Box A^M_0=-\frac t\tau \tau^{-3} (|a_-(y)|^2-|a_+(y)|^2)
\end{equation*}
we can also determine $$\pa_t^2 A^M_0|_{t=2}=\frac t\tau \tau^{-3} (|a_-(y)|^2-|a_+(y)|^2)|_{t=2}=2(4-|x|^2)^{-2}(|a_-(x/2)|^2-|a_+(x/2)|^2).$$
Then we have $(\pa^\mu A^M_\mu)|_{t=2}=0$, $\pa_t(\pa^\mu A^M_\mu)|_{t=2}=-\pa_t^2 A^M_0|_{t=2}$ which is compactly supported in $|x|\leq 2$. However, by strong Huygen's principle, the effect of this initial data part will only be present when $-4\leq q\leq 0$. Therefore, we have $\pa^\mu A^M_\mu=0$ when $q<-4$.
Then using $\pa_\mu q=L_\mu$ and \eqref{convergencerateZ}, we have
\begin{equation*}
    \begin{split}
        \pa^\mu\az_\mu&=\pa^\mu A^M_\mu \mathbf{1}_{|q|\leq 4}+\pa^\mu \left(\chi(\frac \qq r)F_\mu(q,\om) r^{-1}\right)+\pa^\mu (\chi_{ex}(q)\, \q_\infty r^{-1} \delta_{0\mu})\\
        &=\left(\pa^\mu (A^M_\mu-r^{-1} F^M_\mu)+\pa^\mu (r^{-1}F^M_\mu)\right)\mathbf{1}_{|q|\leq 4}+\pa^\mu \left(\chi(\frac\qq r) F_\mu(q,\om)\right) r^{-1}+\chi_{ex}'(q)\q_\infty r^{-1}\\
        &\quad +O(\varepsilon\langle t+r\rangle^{-2})\\
        &=\left(\pa_q F^M_L\right) r^{-1}+\pa_q \left(\chi(\frac\qq r) F_L\right) r^{-1}+\chi_{ex}'(q) \q_\infty r^{-1}+O(\varepsilon\langle t+r\rangle)^{-2}).
    \end{split}
\end{equation*}
Therefore we derive the exact condition we need to impose on the scattering data:
\begin{equation}
    \pa_q \left(F^M_L+\chi(\frac{\qq}{r})F_L+\chi_{ex}(q) \q_\infty\right)=0.
\end{equation}
Note that $F_\mu$ decays to zero as $\qq\rightarrow \infty$, and $F^M_L(q,\om)=\q_\infty$ whenever $q\leq -4$. Therefore, the quantity in the parenthesis should always be $\q_\infty$. In view of the support of $F^M_L-\q_\infty$ and $\chi_{ex}(q)$, the term $\chi(\frac\qq r)F_L$ must vanish at points where $\chi(\frac\qq r)$ is not $1$. So one can drop the cutoff function factor $\chi(\frac\qq r)$ above, and it must hold that
\begin{equation}\label{admissiblecondition}
    F_L=\q_\infty-\q_\infty \chi_{ex}(q)-F^M_L.
\end{equation}

The condition ensures that $\pa^\mu \az_\mu$ decays like $(t+r)^{-2}$. Now we consider the case with vector fields. Towards null infinity we have
\begin{equation*}
    \begin{split}
        rZ(A^M_L)&=rA^M_\mu (ZL^\mu)+r(ZA^M_\mu)L^\mu\\
        &\rightarrow F^M_\mu(q,\om)(ZL^\mu)+r(Z(r^{-1}F^M_\mu(q,\om)))L^\mu=rZ(r^{-1}F_\mu(q,\om)L^\mu)\\
        &=rZ(r^{-1}\q_\infty)
    \end{split}
\end{equation*}
and from the previous subsection, the convergence rate is $O_q(\varepsilon r^{-1})$. Using $\pa^\mu A^M_\mu=0$ when $t-r>4$ again, we have
\begin{equation*}
    \begin{split}
        Z \pa^\mu\az_\mu&=Z \pa^\mu A^M_\mu \mathbf{1}_{|q|\leq 4}+Z\pa^\mu \left(\chi(\frac\qq r)F_\mu(q,\om) r^{-1}\right)+Z\pa^\mu (\chi_{ex}(q)\, \q_\infty r^{-1} \delta_{0\mu})\\
        &=\left(Z\pa^\mu (A^M_\mu-F^M_\mu r^{-1})+Z\pa^\mu (F^M_\mu r^{-1})\right)\mathbf{1}_{|q|\leq 4}+\left(Z\pa^\mu (F_\mu \chi(\frac \qq r))\right) r^{-1}+\pa^\mu\left(\chi(\frac\qq r) F_\mu\right) Z(r^{-1})\\
        &\quad +Z(\chi_{ex}'(q)\q_\infty r^{-1})+O(\varepsilon\langle t+r\rangle^{-2})\\
        &=\pa_q F^M_L Z(r^{-1})+Z(\pa_q F^M_L)r^{-1}+\left(Z \pa_q (F_L \chi(\frac\qq r))\right) r^{-1}+\pa_q(\chi(\frac\qq r) F_L) Z(r^{-1})\\
        &\quad +Z(\chi_{ex}'(q)) \q_\infty r^{-1}+\chi_{ex}'(q)\q_\infty Z(r^{-1})+O(\varepsilon\langle t+r\rangle^{-2})\\
        &=Z(r^{-1})\left(\pa_q F^M_L+\pa_q(\chi(\frac \qq r) F_L)+\q_\infty \chi_{ex}'(q)\right)\\
        &\quad +Z\left(\pa_q F^M_L+\pa_q(\chi(\frac \qq r) F_L)+\q_\infty \chi_{ex}'(q)\right)r^{-1}+O(\langle t+r\rangle^{-2})=O(\varepsilon\langle t+r\rangle^{-2}).
    \end{split}
\end{equation*}
One can apply more vector fields and derive similar estimates. In conclusion, once the scattering data satisfies the admissble condition \eqref{admissiblecondition}, we have 
\begin{equation*}
    Z^I \pa^\mu \az_\mu=O(\varepsilon\langle t+r\rangle^{-2}).
\end{equation*}



\subsection{Estimates of the approximate solutions}
In this subsection we derive estimates of the approximate solutions. First we have (For $A^M_\mu$, we can commute its wave equation with vector fields to derive the estimate)
\begin{equation}
    \begin{split}
        |(Z,S)^I \az_\mu|&\lesssim |Z^I A^M_\mu|+|Z^I(\chi(\qq/r) (F(q,\om) r^{-1}+F^{(1)}(q,\om) r^{-2}))|+|Z^I(\chi_{ex}(q) \q_\infty r^{-1} \delta_{0\mu})|\\
        &\lesssim \varepsilon\langle t+r\rangle^{-1}.
    \end{split}
\end{equation}
Using the definition of $\phiz$, we also have ($\Om$ being rotation or boosts)
\begin{equation}
    |\pa^I \Om^J \phiz|\lesssim \tau^{-\frac 32-|I|} (\ln\tau)^{|J|} (1-|y|^2)^{-\frac{|I|}2}. 
\end{equation}

In view of the definition of the second order approximate radiation field $F^{(1)}(q,\om)$, as in \cite{LindbladSchlue1,CL}, we have the following estimate
\begin{equation*}
    |\Box\Big(\chi(\frac\qq r) \big(\frac{F_\mu(q,\om)}{r}+\frac{F_\mu^{(1)}(q,\om)}{r^2}\big)\Big)|\lesssim \varepsilon\langle t+r\rangle^{-4}\langle q\rangle^{\gamma}.
\end{equation*}
Then since $\chi_{ex}(q)\q_\infty r^{-1}\delta_{0\mu}$ are exact solutions to the linear wave equation, we obtain the estimate
\begin{equation}
    |\Box Z^I(\az_\mu)-Z^I(\frac{x_\mu}\tau \tau^{-3}(|a_+(y)|^2-|a_-(y)|^2)|\lesssim \varepsilon\langle t+r \rangle^{-4}\langle q\rangle^\gamma.
\end{equation}

Now we turn to Klein-Gordon field. We have
\begin{equation*}
    \begin{split}
        -\Box\phiz+\phiz&-2i(\az)^\mu\pa_\mu\phiz=\tau^{-\frac 32}(\pa_\tau^2+1)(\tau^\frac 32 \phiz)-2i(\az)^\tau\pa_\tau\phiz+R_{tan}\\
        &=\tau^{-\frac 32}(\pa^2_\tau+1)(e^{i\tau+i\int (\az)^\tau d\tau}a_+(y)+e^{-i\tau-i\int(\az)^\tau d\tau} a_-(y))\\
        &\quad -2i(\az)^\tau \tau^{-\frac 32}(ie^{i\tau+i\int (\az)^\tau d\tau}a_+(y)-ie^{-i\tau-i\int(\az)^\tau d\tau} a_-(y))+R_{good,1}+R_{tan}\\
        &=\tau^{-\frac 32} (-2(\az)^\tau e^{i\tau+i\int (\az)^\tau d\tau} a_+(y)-2(\az)^\tau e^{-i\tau-i\int (\az)^\tau d\tau}a_-(y))\\
        &\quad -2i(\az)^\tau \tau^{-\frac 32}(ie^{i\tau+i\int (\az)^\tau d\tau}a_+(y)-ie^{-i\tau-i\int(\az)^\tau d\tau} a_-(y))\\
        &\quad +R_{\tan}+R_{good,1}+R_{good,2}+R_{good,3}\\
        &=R_{\tan}+R_{good,1}+R_{good,2}+R_{good,3}.
    \end{split}
\end{equation*}
where
\begin{equation*}
    \begin{split}
        R_{tan}&=-\tau^{-2}(\triangle_y \phiz+\frac 34\phiz)-2i(\az)^{y_i}\pa_{y_i}\phiz,\\
        R_{good,1}&=2((\az)^\tau)^2 \tau^{-\frac 32} (e^{i\tau+i\int (\az)^\tau d\tau}a_+(y))-2((\az)^\tau)^2 \tau^{-\frac 32} (e^{-i\tau-i\int (\az)^\tau d\tau}a_-(y)),\\
        R_{good,2}&=-((\az)^\tau)^2\phiz,\\
        R_{good,3}&=i\tau^{-\frac 32}\pa_\tau(\az)^\tau e^{i\tau+i\int (\az)^\tau d\tau}a_+(y)-i\tau^{-\frac 32}\pa_\tau(\az)^\tau e^{-i\tau-i\int (\az)^\tau d\tau}a_-(y).
    \end{split}
\end{equation*}
These remainder terms enjoy good decay properties. We have
\begin{equation}
    |Z^I (R_{tan}+R_{good,1}+R_{good,2}+R_{good,3})|\lesssim \varepsilon\tau^{-\frac 72}(1-|y|^2)^{\a}.
\end{equation}

\def\HH{\widetilde H_\tau}

\subsection{Backward energy estimate}
We are now ready to derive the backward energy estimate which constructs the solution. This part is very similar to \cite{CL}. We consider the following system 
\begin{equation}
    \Box (\az_{\mu}+v_{\mu})=\Im\left((\phiz+w)\overline{(\pa_\mu(\phiz+w)+i(\az_\mu+v_{\mu})(\phiz+w))}\right),
\end{equation}
\begin{equation}
    -\Box(\phiz+w)+(\phiz+w)=2im^{\mu\nu}(\az_\mu+v_{\mu})\pa_\nu (\phiz+w)-m^{\mu\nu} (\az_\mu+v_{\mu}) (\az_\nu+v_{\nu}) (\phiz+w),
\end{equation}
i.e., let $(\az_\mu+v_\mu,\phiz+w)$ solve the reduced system.
In view of the properties of the approximate solution, this gives
\begin{equation}\label{eqofvT}
    \begin{split}
        \Box v_{\mu}&=\Box(\chi (F_\mu r^{-1}+F_\mu^{(1)} r^{-2}))+R_\mu\\&+\Im(w\overline{\pa_\mu w}+w\overline{w\cdot iv_{\mu}})
        +\Im(\phiz\overline{\pa_\mu w}+w\overline{\pa_\mu\phiz})-|\phiz|^2 v_\mu+\Im(\phiz\overline{i\az_\mu w})\\
        &+\Im(w\overline{i\az\phiz})
        +\Im(\phiz\overline{iv_\mu w})-|w|^2\az_\mu+\Im(w\overline{iv_\mu \phiz}).
    \end{split}
\end{equation}
\begin{equation}\label{eqofwT}
    \begin{split}
        -\Box w+w&=
        2im^{\mu\nu} v_{\mu}\pa_\nu w-m^{\mu\nu}v_{\mu} v_{\nu}w+R\\
        &+2i ((\az)^\mu \pa_\mu w)+2i v^\mu \pa_\mu\phiz-2(\az)^\mu v_\mu (\phiz+w)-v^\mu v_\mu \phiz.
    \end{split}
\end{equation}
where $R_\mu=\Im(\phiz\overline{i\az_\mu\phiz})+\Im(\phiz\overline{\pa_\mu\phiz})-\frac{x_\mu}\tau \tau^{-3}(|a_+(y)|^2-|a_-(y)|^2)$, $R=\Box\phiz-\phiz+2i (\az)^\mu \pa_\mu\phiz$. These remainders are independent of $v_\mu$ and $w$, and we have good estimates of them. We have already obtained 
\begin{equation*}
    |Z^I R|\lesssim \varepsilon\tau^{-\frac 72}(1-|y|^2)^{\a}.
\end{equation*}
We also have $$R_\mu=-|\phiz|^2 \az_\mu+\Im\left(\phiz \overline{\pa_\mu\left(\tau^{-\frac 32} e^{i\int (\az)^\tau d\tau}a_+(y)\right)e^{i\tau}+\pa_\mu\left(\tau^{-\frac 32} e^{-i\int (\az)^\tau d\tau}a_-(y)\right)e^{-i\tau}}\right).$$ Recall that $\pa_t=\frac t\tau\pa_\tau-t^{-1}y\cdot\nab_y$, $\pa_i=-\frac{x_i}\tau\pa_\tau+t^{-1}\pa_{y_i}$. One can then express $y$-derivatives using vector fields as \eqref{laplaciany} to get
\begin{equation}
    |Z^I R_\mu|\lesssim 
    \varepsilon\tau^{-4} (1-|y|^2)^{2\a-\frac 12}.
\end{equation}

We consider $(v_{\mu,T},w_T)$ solving the equations above, with the whole right hand side multiplied by $\tilde\chi(t/T)$. The cutoff $\tilde\chi(s)$ is non-increasing with $\tilde\chi(s)=1$ when $s\leq 1/4$, and $\tilde\chi(t/T)=0$ when $s\geq 1/2$. We solve the equations backward, and let the data at $t=T$ vanish. We aim to show the limits of $v_{\mu,T}$ and $w_T$ exist as $T\rightarrow \infty$.

First we note that the vector field applied to the cutoff essentially gives the same thing. 
\begin{lem}
    $|Z^I(\tilde\chi(t/T))|\lesssim \tilde\chi(t/2T)$ in the relevant region. By relevant region we mean the place where the right hand side of \eqref{eqofvT} and \eqref{eqofwT} are nonzero.
\end{lem}
To see this, notice that all terms in the equations which are independent of $v_{\mu,T}$ and $w_T$, are supported in $\{t\geq r/5\}$. Then since $v_{\mu,T}$ and $w_T$ are zero near $t=T$, by finite speed of propagation, we see that $v_{\mu,T}$ and $w_T$ will be supported in $\{t+r\leq 6T\}$. Therefore the relevant region is contained in $\{t+r\leq 6T\}$, and the lemma follows by the expression of vector fields.


\subsubsection{Bootstrap assumptions, $L^2$ estimate of approximate solutions}
As in \cite{CL}, we work on the foliation with the part when $t-r\geq r^\frac 12$ is a truncated hyperboloid $\widetilde H_\tau$, and is extended to the exterior by a constant time slice denoted by $\Sigma_\tau^e$. One can compute that on $\Sigma_\tau^e$, $t\approx \tau^\frac 43$. One also have $dt\approx \tau^\frac 13 d\tau$ in the exterior. Then we have the energy estimate
\begin{equation*}
    \begin{split}
        E_{w}(Z^I v_{\mu,T},\tau_1)^\frac 12& \lesssim E_{w}(Z^I v_{\mu,T},\tau_2)^\frac 12+\int_{\tau_1}^{\tau_2} ||Z^I(\Box v_{\mu,T})||_{L^2(\widetilde H_\tau)}+\tau^\frac 13 ||Z^I(\Box v_{\mu,T})||_{L^2(\Sigma_\tau^e)} d\tau,\\
        E_{KG}(Z^I w_T,\tau_1)^\frac 12\lesssim &\, E_{KG}(Z^I w_T,\tau_2)^\frac 12+\int_{\tau_1}^{\tau_2} ||Z^I(-\Box+1) w_T||_{L^2(\widetilde H_\tau)}+\tau^\frac 13 ||Z^I(-\Box+1)w_T||_{L^2(\Sigma_\tau^e)} d\tau,
    \end{split}
\end{equation*}
where the energy reads
\begin{equation}
\begin{split}
    E_w(\psi,\tau)&=\int_{\widetilde H_\rho} |(\tau/t) \pa_t \psi|^2+|t^{-1}\Om_{0i} \psi|^2 dx+\int_{\Sigma_\rho^e} |\pa \psi|^2 dx,\\
    E_{KG}(\psi,\tau)=\int_{\widetilde H_\rho} &|(\tau/t) \pa_t \psi|^2+|t^{-1}\Om_{0i} \psi|^2+|\psi|^2 dx+\int_{\Sigma_\rho^e} |\pa \psi|^2+|\psi|^2 dx.
    \end{split}
\end{equation}
We also define the higher order version $E_{w,k}(\psi,\tau)$ by $E_{KG,k}(\psi,\tau)$ by applying at most $k$ vector fields in $\mathcal Z$ to $\psi$.

We have derived estimates of the approximate solutions above. Now we need estimate them at $L^2$ level. Denote $R'_\mu=R_\mu+\Box(\chi(\qq/r)(F_\mu r^{-1}+F_\mu^{(1)} r^{-2}))$. Then
\begin{equation}
\begin{split}
    ||Z^I R'_\mu||_{L^2(\HH)}&\lesssim ||\varepsilon\langle t+r\rangle^{-4}\langle q\rangle^\gamma||_{L^2(\HH)}\lesssim \varepsilon\tau^{2\gamma}\left(\int_{\tau/\sqrt 3}^\infty \langle t+r\rangle^{-8-2\gamma} r^2 dr\right)^\frac 12\lesssim \varepsilon\tau^{-\frac 52+\gamma},\\
    ||Z^I R'_\mu||_{L^2(\Sigma^e_\tau)}&\lesssim ||\varepsilon\langle t+r\rangle^{-4}\langle q\rangle^\gamma||_{L^2(\Sigma^e_\tau)}\lesssim \varepsilon\left(\int_{-t\leq t-r\leq 2t^\frac 12} \frac{r^2}{\langle t+r\rangle^8} \langle q\rangle^{2\gamma} d\om dr\right)^\frac 12\\
    &\lesssim \varepsilon\left(\int_{-2t^\frac 12}^t t^{-6} \langle q\rangle^{2\gamma} dq\right)^\frac 12\lesssim \varepsilon t^{-\frac 52+\gamma}\lesssim \varepsilon\tau^{-\frac{10}3+\gamma},\\
    ||Z^I R||_{L^2(\HH)}&\lesssim ||\varepsilon\langle t+r\rangle^{-\frac 72}||_{L^2(\HH)}\lesssim \varepsilon\left(\int_0^{2\tau^\frac 43} \langle t+r\rangle^{-7} r^2 dr\right)^\frac 12\lesssim \varepsilon\tau^{-4}\left(\int_0^{2\tau^\frac 43} \langle t+r\rangle^{-3} r^2 dr\right)^\frac 12\\
    &\lesssim \varepsilon\tau^{-2}(\ln\tau)^\frac 12,\\
    ||Z^I R||_{L^2(\Sigma^e_\tau)}&\lesssim ||\varepsilon\langle t+r\rangle^{-\frac 72}||_{L^2(\Sigma^e_\tau)}\lesssim \varepsilon(t^{-7}\c t^\frac 12)^\frac 12\lesssim \varepsilon t^{-\frac {13}4}\lesssim \varepsilon\tau^{-\frac{13}3}.
\end{split}
\end{equation}


We now make the bootstrap assumption that for all $\tau\geq T^*$,
\begin{equation}
    E_{w,k}(v_T,\tau)^\frac 12 \leq C_b\varepsilon\tau^{-\frac 32+\gamma+k\delta},\quad E_{KG,k}(w_T,\tau)^\frac 12\leq C_b \varepsilon \tau^{-1+\gamma+k\delta}.
\end{equation}
This clearly holds for all $\tau\geq T$. We will improve the bounds and hence show that one can take $T^*=1$.


\subsubsection{Decay estimates}
With bootstrap assumptions we can obtain decay estimates of $v_{\mu,T}$ and $w_T$. 
Using Sobolev embeddings, we have
\begin{equation*}
    |Z^I v_{\mu,T}|\lesssim C_b \varepsilon t^{-\frac 12+(|I|+2)\delta}\tau^{-\frac 32+\gamma+(|I|+2)\delta},\quad |Z^I w_T|\lesssim C_b \varepsilon t^{-\frac 32} \tau^{-1+\gamma+(|I|+2)\delta},\quad \text{when }t-r\leq r^\frac 12,
\end{equation*}
\begin{equation*}
    |Z^I v_{\mu,T}|\lesssim C_b\varepsilon r^{-\frac 12}\tau^{-\frac 32+\gamma+(|I|+2)\delta},\quad |Z^I w_T|\lesssim C_b \varepsilon r^{-1} \tau^{-1+\gamma+(|I|+2)\delta}, \quad\text{when }t-r\leq r^\frac 12.
\end{equation*}
(See proof of the embeddings in \cite[Lemma 6.6, 6.8]{CL}.)
To derive the $L^2$ bound of $v_{\mu,T}$ themselves, we can use the Hardy-type estimate (\cite[Lemma 6.9]{CL}) to get
\begin{equation*}
    ||r^{-1}Z^I v_{\mu,T}||_{L^2(\HH)}+||r^{-1} Z^I v_{\mu,T}||_{L^2(\Sigma^e_s)}\leq 8E_w(v_{\mu,T},\tau)^\frac 12\lesssim C_b \varepsilon \tau^{-\frac 32+\gamma+|I|\delta}.
\end{equation*}

\subsubsection{Energy estimates}
We start with the energy estimate of the wave (gauge potential) components. For shorthand notations, we write $v_\mu=v_{\mu,T}$, $w=w_T$. Recall the equation of $v_{\mu,T}$, i.e., \eqref{eqofvT} with the right hand side multiplied by $\tilde\chi(t/T)$. There are many terms on the right hand side, but starting from $-|\phiz|^2 v_\mu$, all terms are lower order terms and can be omitted. Also, using the decay of $a_\pm(y)$, the terms with $\phiz$ (including $R_\mu$) contribute little in the exterior, and thus can also be omitted. For $|I|\leq k$, the estimate reads
\begin{multline*}
    E_{w}(Z^I v_\mu,\tau_1)^\frac 12\lesssim E_{w}(Z^I v_\mu,\tau_2)^\frac 12+\int_{\tau_1}^{\tau_2} (||Z^I R'_\mu||_{L^2(\HH)}
    +||Z^I(\Im(w\overline{\pa_\mu w}))||_{L^2(\HH)}\\
    +||Z^I(\Im(\phiz\pa_\mu w+w\pa_\mu\phiz))||_{L^2(\HH)}
    +\int_{\tau_1}^{\tau_2} \tau^{\frac 13} (||Z^I R'_\mu||_{L^2(\Sigma^e_\tau)}+||Z^I(\Im(w\overline{\pa_\mu w}))||_{L^2(\Sigma^e_\tau)} d\tau\\
    \lesssim \varepsilon\tau_1^{-\frac 32+\gamma+k\delta}+\int_{\tau_1}^{\tau_2} \varepsilon \tau^{-\frac 32} E_{KG,k}(w,\tau)^\frac 12+\varepsilon \tau^{-\frac 32}\ln\tau\, E_{KG,k-1}(w,\tau)^{\frac 12}+C_b^2 \varepsilon^2 \tau_1^{-2+\frac 34\gamma+(k+2)\delta} d\tau\\
    +\int_{\tau_1}^{\tau_2} C_b \varepsilon\tau^{-2+\gamma}E_{KG,k}(w,\tau)^\frac 12 d\tau \lesssim (\varepsilon+C_b^2\varepsilon^2)\tau_1^{-\frac 32+\gamma+k\delta}.
\end{multline*}
Now we turn to the Klein-Gordon equation. Again, one can ignore terms involving $\phiz$ (including $R$) in the exterior. We have
\begin{multline*}
    E_{KG}(\tau_1,Z^I w)^\frac 12\lesssim E_{KG}(\tau_2,Z^I w)^\frac 12+\int_{\tau_1}^{\tau_2} ||Z^I R||_{L^2(\HH)}+||Z^I((\az)^\mu \pa_\mu w)||_{L^2(\HH)}\\
    +||Z^I(v\cdot \pa w)||_{L^2(\HH)}+||Z^I(v^\mu \pa_\mu\phiz)||_{L^2(\HH)}
    +\tau^\frac 13\left(||Z^I((\az)^\mu \pa_\mu w)||_{L^2(\Sigma^e_\tau)}+||Z^I(v\cdot \pa w)||_{L^2(\Sigma^e_\tau)}\right) d\tau\\
    \lesssim \int_{\tau_1}^{\tau_2} \varepsilon\tau^{-2}\ln\tau+\varepsilon\tau^{-1}E_{KG,k}(w,\tau)^\frac 12+\varepsilon\tau^{-\frac 32}E_{w,k}(v_\mu,\tau)^\frac 12\\
    +\varepsilon\tau^{-\frac 32}\ln\tau E_{w,k-1}(v_\mu,\tau)^\frac 12+C_b^2\varepsilon^2\tau^\frac 13 \tau^{-\frac 32+\gamma+k\delta} \tau^{-1+\gamma+2\delta}d\tau
    \\
    \lesssim \varepsilon\tau_1^{-1}\ln\tau_1+C_b\varepsilon^2 \tau_1^{-1+\gamma+k\delta}+C_b^2\varepsilon^2\tau_1^{-\frac 76+2\gamma+(k+2)\delta},
\end{multline*}
where we used $|Z^I((\az)^\mu \pa_\mu w)|\lesssim \tau^{-1} \frac \tau t |\pa Z^J w|$ on $\HH$, and $\tau^\frac 13|Z^I((\az)^\mu \pa_\mu)|\lesssim \tau^\frac 13 t^{-1}|\pa Z^J w|\lesssim \tau^{-1}|\pa Z^J w|$ on $\Sigma^e_\tau$. When more vector fields fall on $v$, we have used the bounds from Hardy estimate, for example
\begin{equation*}
    \tau^\frac 13 ||Z^I v\c \pa w||_{L^2(\Sigma^e_\tau)}\lesssim \tau^\frac 13 ||r^{-1} Z^{I_1}v_\mu||_{L^2(\Sigma^e_\tau)} ||rw||_{L^\infty(\Sigma^e_\tau)}\lesssim C_b^2\varepsilon^2\tau^{\frac 13} \tau^{-\frac 32+\gamma+k\delta} \tau^{-1+\gamma+2\delta}.
\end{equation*}
Therefore for given $\gamma<1/6$, we improve the boostrap bounds, so the bootstrap bounds on $v_{\mu,T}$ and $w_T$ hold for all $\tau\geq T^*$.

\subsection{Taking the limit}
We want to show that the limit as $T\rightarrow \infty$ exists.
Let $T_2>T_1$. We denote $v_{\mu,i}:=v_{\mu,T_i}$, $w_i:=w_{T_i}$. Consider the difference $\hat v_\mu:=v_{\mu,2}-v_{\mu,1}$ and $\hat w:=w_2-w_1$ ($i=1,2$). We have
\begin{equation*}
    -\Box \hat v_\mu=\tilde\chi(t/T_2)N_{\mu}(v_{\mu,2},w_2)
    -\tilde\chi(t/T_1)N_\mu(v_{\mu,1},w_1),
\end{equation*}
\begin{equation*}
    -\Box \hat w+\hat w=\tilde\chi(t/T_2)N(v_{\mu,2},w_2)
    -\tilde\chi(t/T_1)N(v_{\mu,1},w_1).
\end{equation*}
Here $N_\mu$ and $N$ are the terms on the right hand side of \eqref{eqofvT} and \eqref{eqofwT} respectively.
Then we get
\begin{equation*}
    -\Box {\hat v}_\mu=(\tilde\chi(t/T_2)-\tilde\chi(t/T_1))N_\mu(v_{\mu,1},w_1)\\
    +\tilde\chi(t/T_2)(N_\mu(v_{\mu,2},w_2)-N_\mu(v_{\mu,1},w_1)),
\end{equation*}
\begin{equation*}
    -\Box {\hat w}+\hat w=(\tilde\chi(t/T_2)-\tilde\chi(t/T_1))N(v_{\mu,1},w_1)\\
    +\tilde\chi(t/T_2)(N(v_{\mu,2},w_2)-N(v_{\mu,1},w_1)).
\end{equation*}

We then consider the energy estimate between $\Sigma_\tau$ and $\Sigma_{T_1}$ with $\tau<T_1$. Note that $v_{\mu,1}$ and $w_1$ vanish near $\Sigma_{T_1}$. We also have established the bounds
\begin{equation}
    E_{KG,k}(w_1,\tau)^\frac 12+E_{KG,k}(w_2,\tau)^\frac 12\lesssim \varepsilon \tau^{-1+\gamma+k\delta},\quad E_{w,k}(v_{\mu,1},\tau)^\frac 12+E_{w,k}(v_{\mu,2},\tau)^\frac 12\lesssim \varepsilon\tau^{-\frac 32+\gamma+k\delta}
\end{equation}
for all $\tau\leq T_2$, as well as the corresponding decay estimates.

Since $\tilde\chi(t/T_2)-\tilde\chi(t/T_1)$ is nonzero only when $\tau\geq\frac{1}{16}T_1^\frac 34$, the integration in $\tau$ in the energy estimate starts at this value. Then the integral of the first term in the energy estimate for both equations decays is $T_1$ since the integrands are essentially the same as the estimates above. 

For the second term, because of the difference form, all terms can be written as a factor times $\hat v_\mu$ or $\hat w$. Using the bounds we established above, we have the following estimates, by noticing that the structures are the same as the estimate above ($|I|\leq N-2$):
\begin{multline*}
    \int_\tau^{T_1} ||Z^I(N_\mu(v_{\mu,2},w_2)-N_\mu(v_{\mu,1},w_1))||_{L^2(\widetilde H_s)}+s^\frac 13 ||Z^I(N_\mu(v_{\mu,2},w_2)-N_\mu(v_{\mu,1},w_1))||_{L^2(\Sigma_s^e)} ds\\
    \lesssim\int_\tau^{T_1} 
    ||Z^I(\pa_\mu\phiz,\phiz)(\hat w,\pa\hat w)||_{L^2(\widetilde H_s)}+s^\frac 13||Z^I(\pa_\mu\phiz,\phiz)(\hat w,\pa\hat w)||_{L^2(\Sigma_s^e)}\\
    +||Z^I((\pa_\mu w_i) (\hat w,\pa_\mu \hat w))||_{L^2(\widetilde H_s)}
    +s^\frac 13 ||Z^I((\pa_\mu w_i) (\hat w,\pa_\mu \hat w))||_{L^2(\Sigma_s^e)}+(\varepsilon s^{-1} E_{w,k}(s,\hat v_\mu)^\frac 12) ds\\
    \lesssim \int_\tau^{T_1} \varepsilon (s^{-\frac 32}+s^{-2+\a+\delta}) E_{KG,k}(s,\hat w)^\frac 12+\varepsilon s^{-\frac 32}(\ln s)\, E_{KG,k-1}(s,\hat\om)^\frac 12+(\varepsilon s^{-1} E_{w,k}(s,\hat v_\mu)^\frac 12)\, ds,
\end{multline*}
and
\begin{multline*}
    \int_\tau^{T_1} ||Z^I(N(v_{\mu,2},w_2)-N(v_{\mu,1},w_1))||_{L^2(\widetilde H_s)}+s^\frac 13 ||Z^I(N(v_{\mu,2},w_2)-N(v_{\mu,1},w_1))||_{L^2(\Sigma_s^e)}\\
    \lesssim\int_\tau^{T_1} ||Z^I(v_{\mu,i}(\hat w,\pa_\mu \hat w))||_{L^2(\widetilde H_s)}+||Z^I((\pa_\mu w_i,w_i)\hat v_{\mu}||_{L^2(\widetilde H_s)}\\
    +s^\frac 13(||Z^I(v_{\mu,i}(\hat w,\pa_\mu \hat w))||_{L^2(\Sigma_\tau^e)}+||Z^I((\pa_\mu w_i,w_i)\hat v_{\mu}||_{L^2(\Sigma_\tau^e)})\\
    +||Z^I(\az\cdot \pa \hat w)||_{L^2(\widetilde H_s)}+s^\frac 13 ||Z^I(\az\cdot \pa \hat w)||_{L^2(\Sigma^e_s)}
    +||Z^I(\pa\phiz\c\hat v_\mu)||_{L^2(\widetilde H_s)}+s^\frac 13 ||Z^I(\pa\phiz\c\hat v_\mu)||_{L^2(\Sigma^e_s)} ds\\
    \lesssim \int_\tau^{T_1} \varepsilon s^{-1}E_{KG,k}(s,\hat w)^\frac 12+\varepsilon (s^{-\frac 12}+s^{-\frac 23+\gamma+2\delta}) E_{w,k}(s,\hat v_\mu)^\frac 12+\varepsilon s^{-\frac 12}(\ln s) E_{w,k-1}(s,\hat v_\mu)^\frac 12\, ds.
\end{multline*}
These estimates are similar to the energy estimate we did above.
Now define $E_k(\tau)^\frac 12=E_{w,k}(\hat v,\tau)^\frac 12+\tau^{\frac 12} E_{KG;k}(\hat v,\tau)^\frac 12$. We have for $\tau\leq T_1$ that
\begin{equation*}
    E_k(\tau)^\frac 12 \lesssim \varepsilon (T_1)^{-\frac 12+\gamma+k\delta}+\varepsilon((T_1)^\frac 34)^{-\frac 12+\gamma+k\delta} +\int_\tau^{T_1} \varepsilon s^{-1} E_k(s)^\frac 12+\varepsilon s^{-1}(\ln s)^k E_{k-1}(s)^\frac 12 ds.
\end{equation*}
When $k=0$, the $E_{k-1}$ term does not appear, so using Gr\"{o}nwall's inequality we have
\begin{equation*}
    E_0(\tau)^\frac 12 \lesssim \varepsilon (T_1/\tau)^{C\varepsilon} (T_1)^{-\frac 38+\frac 34\gamma}\rightarrow 0\quad\text{as }T_1\rightarrow \infty.
\end{equation*}
Then it is straightforward to show by induction that
\begin{equation*}
    E_k(\tau)^\frac 12\lesssim \varepsilon (T_1/\tau)^{C\varepsilon} (T_1)^{-\frac 38+\frac 34\gamma+\frac 34 k\delta},\quad k\leq N-2.
\end{equation*}
Then by Sobolev embeddings, we have for $|I|\leq N-4$ that
\begin{equation*}
    \sup_{\tau(t,x)\leq T_1}|Z^I \hat w|+|Z^I \hat v|\lesssim \varepsilon (T_1)^{-\frac 38+\frac 34\gamma+\frac 34(k+2)\delta+C\varepsilon}
\end{equation*}
which converges uniformly to zero as $T_2>T_1\rightarrow \infty$. This shows the existence of the limit $w=\lim_{T\rightarrow\infty} w_T$ and $v_\mu=\lim_{T\rightarrow\infty} v_{\mu,T}$, and that $(\az_\mu+v_{\mu},\phiz+w)$ gives a solution of the reduced mMKG system, with the asymptotic behavior being exactly the one given by the scattering data.

\subsection{The original system}\label{sectionoriginalsystem}
In this subsection, we show that the solution we obtained above also satisfies the Lorenz gauge condition, hence solves the original mMKG system. Since $\az_\mu+v_{\mu,T}$ and $\phiz+w_T$ solve the reduced system for each $T$, using \eqref{eqoflambda}, we have that $\lambda_T=\pa^\mu(\az_\mu+v_{\mu,T})$ satisfies the equation
\begin{equation}
    \Box\lambda_T=|\phiz+w_T|^2\lambda_T.
\end{equation}
Recall we have $Z^I\lambda_T=Z^I\pa^\mu (\az_\mu+v_{\mu,T})=O(\varepsilon\langle t+r\rangle^{-2})+Z^I\pa^\mu v_{\mu,T}$.
One can then do the energy estimate
\begin{equation*}
    E_{w}(Z^I\lambda_T,\tau)^\frac 12\lesssim E_w(Z^I\lambda_T,T)^\frac 12+\int_\tau^T ||Z^I(|\phiz+w_T|^2\lambda_T)||_{L^2(\HH)}
    +s^\frac 13||Z^I(|\phiz+w_T|^2\lambda_T)||_{L^2(\Sigma^e_s)}ds.
\end{equation*}
We have shown that
\begin{equation*}
    |\pa Z^I \pa^\mu A_\mu^{(0)}|\lesssim \varepsilon\langle t+r\rangle^{-2},
\end{equation*}
and we know that $v_{\mu,T}$ vanish near $\tau=T$. Therefore we have
\begin{equation*}
    \begin{split}
        E_{w} (Z^I\lambda_T,T)^\frac 12
        &\lesssim (\int_{T\leq t,\, 0\leq r\leq 2T^\frac 43} \varepsilon^2(t+r)^{-4} dx)^\frac 12+(\int_{t=T^\frac 43,r\geq t/2} \varepsilon^2(t+r)^{-4} dx)^\frac 12+E_w(Z^I\pa v_{\mu,T},T)^\frac 12\\
        &\lesssim \varepsilon T^{-\frac 12},  
    \end{split}
\end{equation*}
so for all $|I|\leq N-2$, the integral can be estimated as
\begin{equation*}
    \begin{split}
        \int_\tau^T ||\, Z^I&(|\phiz+w_T|^2\lambda_T)||_{L^2(\widetilde H_s)}+s^\frac 13||\, Z^I(|\phiz+w_T|^2\lambda_T)||_{L^2(\Sigma^e_s)}d s\\
        &\lesssim \sum_{|I_1|\leq |I|}\int_\tau^T ||\varepsilon t^{-3}Z^{I_1}\lambda_T||_{L^2(\widetilde H_s)}+\varepsilon\tau^\frac 13 \tau^{-2+2\gamma}r^{-1}||r^{-1} Z^{I_1}\lambda_T||_{L^2(\Sigma^e_s)} ds\\
        &\lesssim \int_\tau^T \varepsilon t^{-2}||r^{-1}Z^{I_1}\lambda_T||_{L^2(\widetilde H_s)}+\varepsilon \tau^{-2}||r^{-1} Z^{I_1} \lambda_T||_{L^2(\Sigma^e_s)} ds\\
        &\lesssim \sum_{|I_1|\leq |I|}\int_\tau^T \varepsilon\tau^{-2} E_w(Z^{I_1} \lambda_T,\tau)^\frac 12 ds.
    \end{split}
\end{equation*}
Then using Gr\"{o}nwall's inequality for the estimate, we get
\begin{equation*}
    E_w(Z^I\lambda_T,\tau)^\frac 12\lesssim \varepsilon T^{-\frac 12}.
\end{equation*}
Therefore, $E_w(Z^I\lambda_T,\tau)^\frac 12$ is bounded by $T^{-1}$ for all $\tau\leq T$. We also have that $\lambda_T$ vanishes as $r\rightarrow 0$ at each time slice. Then, letting $T\rightarrow \infty$, we see that the Lorenz gauge condition is satisfied everywhere for the solution we get.

\subsection{The case when the radiation fields decay slower}
We have made the assumption that $F_\mu(q,\om)$ is decaying at the rate of $\langle q\rangle^{-1+\gamma}$ where $\gamma<1/6$. In this section we discuss the case when they decay even less. They still have to, of course, satisfy the condition \eqref{admissiblecondition}.

One can see from above that, in some sense, this part of radiation fields, compared with the parts governed by the Klein-Gordon source and the charge, corresponds to the part of the wave components satisfying the linear wave equation. In fact, we can first find linear solutions $A_\mu^{Linear}$ that scatter to the radiation fields $F_\mu(q,\om)$. This is a linear scattering problem, and \cite{LindbladSchlue1} it was shown that for all $\gamma<1/2$ one can find linear solutions $A_\mu^{Linear}$ satisfying
\begin{equation*}
    |(Z,S)^I \left(A_\mu^{Linear}-\chi(\qq/r)(F_\mu(q,\om)r^{-1}+F^{(1)}_\mu(q,\om))\right)|\lesssim \varepsilon\langle t+r\rangle^{-1} t^{-\gamma}.
\end{equation*}
Then we can modify our definition of the approximate solution $\az$:
\begin{equation*}
    \az_\mu=A^M_\mu+A^{Linear}_\mu+\q_\infty\chi_{ex}'(q)r^{-1}\delta_{0\mu}.
\end{equation*}
Since $-\Box A_\mu^{Linear}=0$, the estimate of perturbations are in fact easier. However, we need to be more careful about the Lorenz gauge condition. The above condition implies $$|Z^I\pa^\mu \left(A_\mu^{Linear}-\chi(\qq/r)(F_\mu(q,\om)r^{-1}+F^{(1)}_\mu(q,\om))\right)|\lesssim \varepsilon \langle t+r\rangle^{-1}t^{-\gamma}(1+|q|)^{-1}\lesssim \varepsilon\langle t+r\rangle^{-1-\gamma},$$ 
which is worse than $O(\varepsilon\langle t+r\rangle^{-2})$, so the new $\pa^\mu \az_\mu$ is only a $O(\varepsilon\langle t+r\rangle^{-1-\gamma})$ term. But if we review the proof in Section \ref{sectionoriginalsystem} above, we see that this decay is still enough for $\gamma<1/2$, and the bound of $E_w(Z^I\lambda_T,T)^\frac 12$ there becomes $\varepsilon T^{-\gamma+\frac 12}$, still allowing us to prove the result.

\bibliographystyle{abbrv}
\bibliography{reference}

\end{document}